\documentclass[a4paper,12pt]{amsart}

\usepackage{bbm}
\usepackage{units}
\usepackage{mathtools}
\usepackage{cancel}
\usepackage{amsmath}
\usepackage{amsfonts}
\usepackage{latexsym}
\usepackage{amssymb}
\usepackage{mathrsfs}
\usepackage{amscd}
\usepackage{hyperref}
\usepackage{soul}
\usepackage{psfrag}
\usepackage{graphicx}
\usepackage{ulem}
\usepackage{fullpage}
\usepackage[all]{xy}
\usepackage{rotating}
\usepackage{url}
\usepackage{color}
\usepackage{enumerate}
\usepackage{cleveref}

\hypersetup{
	colorlinks=true, %set true if you want colored links
	linkcolor=red, %choose some color if you want links to stand out
	citecolor=blue,
}

\usepackage{dsfont}
\usepackage{tikz}
\usetikzlibrary{arrows,backgrounds,arrows.meta,decorations.markings}
\usepackage{tikz-cd}

\numberwithin{equation}{section}
\numberwithin{figure}{section}

\theoremstyle{plain}
\newtheorem{thm}{Theorem}[section]

\theoremstyle{plain}
\newtheorem{lem}[thm]{Lemma}

\theoremstyle{remark}
\newtheorem{rem}[thm]{Remark}

\theoremstyle{plain}

\theoremstyle{definition}
\newtheorem{defn}[thm]{Definition}

\theoremstyle{definition}
\newtheorem{ex}[thm]{Example}

\theoremstyle{definition}

\theoremstyle{plain}
\newtheorem{prop}[thm]{Proposition}

\theoremstyle{plain}

\theoremstyle{definition}

\theoremstyle{plain}

\newcommand{\comments}[1]{}

\newcommand{\rab}{\rangle}

\newcommand{\lab}{\langle}

\newcommand{\mcal}{\mathcal}

\newcommand{\C}{\mathbb{C}}

\newcommand{\vlon}{\varepsilon}

\title{Quantum Sets of Compact Quantum Groups}
\author{Mainak Ghosh}
\begin{document}
	\maketitle
	\global\long\def\vlon{\varepsilon}
	\global\long\def\bt{\bowtie}
	\global\long\def\ul#1{\underline{#1}}
	\global\long\def\ol#1{\overline{#1}}
	\global\long\def\norm#1{\left\|{#1}\right\|}
	\global\long\def\os#1#2{\overset{#1}{#2}}
	\global\long\def\us#1#2{\underset{#1}{#2}}
	\global\long\def\ous#1#2#3{\overset{#1}{\underset{#3}{#2}}}
	\global\long\def\t#1{\text{#1}}
	\global\long\def\lrsuf#1#2#3{\vphantom{#2}_{#1}^{\vphantom{#3}}#2^{#3}}
	\global\long\def\tr{\triangleright}
	\global\long\def\tl{\triangleleft}
	\global\long\def\cc90#1{\begin{sideways}#1\end{sideways}}
	\global\long\def\turnne#1{\begin{turn}{45}{#1}\end{turn}}
	\global\long\def\turnnw#1{\begin{turn}{135}{#1}\end{turn}}
	\global\long\def\turnse#1{\begin{turn}{-45}{#1}\end{turn}}
	\global\long\def\turnsw#1{\begin{turn}{-135}{#1}\end{turn}}
	\global\long\def\fusion#1#2#3{#1 \os{\textstyle{#2}}{\otimes} #3}
	
	\global\long\def\abs#1{\left|{#1}\right	|}
	\global\long\def\red#1{\textcolor{red}{#1}}

\begin{abstract}	
	Q-system completion can be thought of as a notion of higher idempotent completion of C*-2-categories. We introduce a notion of quantum bi-elements, and study Q-system completion in the context of compact quantum groups. We relate our notion of quantum bi-elements to already known notions of quantum sets and quantum functions, and provide a description of Q-system completion of the C*-2-category of compact quantum groups using quantum bi-elements.
\end{abstract}

\section{Introduction}\label{introduction}

Quantum information theory can be thought of as the study of information processing in a quantum setting, and is the quantum analog of classical (Shannon) information theory. Usually only finite dimensional systems or infinitely many copies of finite systems are considered. Comparatively, little work has been done on quantum systems with infinitely many degrees of freedom. In \cite{N18}, the rich structure of Subfactor theory has been leveraged to study quantum information tasks in infinite systems.

\vspace*{2mm}

In \cite{M03}, a Q-system which is a unitary version of a separable Frobenius algebra object in a C*-tensor category or a C*-2-category, is exhibited as an alternative axiomatization of the standard invariant of a finite index subfactor \cite{O88,P95,J99}. In the context of C*-2-categories, a Q-system is a 1-cell $_bQ_b \in \mcal C_1(b,b)$ along with two 2-cells $m : Q \boxtimes Q \to Q$ (multiplication) and $i : 1_b \to Q$ (unit), which are graphically denoted by the following:
\[m = \raisebox{-6mm}{
	\begin{tikzpicture}%multiplication
		\draw[red,in=90,out=90,looseness=2] (-0.5,0.5) to (-1.5,0.5);
		%\draw[red] (-1.2,0.5) to (-1.2,-.2); 
		%\draw[red] (-.5,0.5) to (-.5,-.2);
		\node at (-1,1.1) {${\color{red}\bullet}$};
		\draw[red] (-1,1.1) to (-1,1.6);
		\node[left,scale=0.7] at (-1,1.4) {$Q$};
		\node[left,scale=0.7] at (-1.6,0.5) {$Q$};
		\node[right,scale=0.7] at (-.5,.5) {$Q$};
\end{tikzpicture}} \ \ \ \ \ \ i = \raisebox{-6mm}{
	\begin{tikzpicture}%unit
		\draw [red] (-0.8,-.6) to (-.8,.6);
		\node at (-.8,-.6) {${\color{red}\bullet}$};
		\node[left,scale=0.7] at (-.8,0) {$Q$};
\end{tikzpicture}} \ \ \ \ \ \ m^* = \raisebox{-6mm}{
	\begin{tikzpicture}%multiplication
		\draw[red,in=-90,out=-90,looseness=2] (-0.5,0.5) to (-1.5,0.5);
		\node at (-1,-.1) {${\color{red}\bullet}$};
		\draw[red] (-1,-.1) to (-1,-.6);
		\node[left,scale=0.7] at (-1,-.4) {$Q$};
		\node[left,scale=0.7] at (-1.6,0.5) {$Q$};
		\node[right,scale=0.7] at (-.5,.5) {$Q$};
\end{tikzpicture}} \ \ \ \ \ \ i^* = \raisebox{-6mm}{
	\begin{tikzpicture}%unit
		\draw [red] (-0.8,-.6) to (-.8,.6);
		\node at (-.8,.6) {${\color{red}\bullet}$};
		\node[left,scale=0.7] at (-.8,0) {$Q$};
\end{tikzpicture}} \]
These $2$-cells satisfy the following:
\[\raisebox{-6mm}{
	\begin{tikzpicture}
		\draw[red,in=90,out=90,looseness=2] (0,0) to (1,0);
		\draw[red,in=90,out=90,looseness=2] (0.5,.6) to (-.5,.6);
		\draw[red] (-.5,.6) to (-.5,0);
		\node at (.5,.6) {$\red{\bullet}$};
		\node at (0,1.2) {$\red{\bullet}$};
		\draw[red] (0,1.2) to (0,1.6);
\end{tikzpicture}}
= \raisebox{-6mm}{\begin{tikzpicture}
		\draw[red,in=90,out=90,looseness=2] (0,0) to (1,0);
		\draw[red,in=90,out=90,looseness=2] (.5,.6) to (1.5,.6);
		\node at (.5,.6) {$\red{\bullet}$};
		\draw[red] (1.5,.6) to (1.5,0);
		\node at (1,1.2) {$\red{\bullet}$};
		\draw[red] (1,1.2) to (1,1.6);
\end{tikzpicture}} \ \t{(associativity)} \ \ \ \ \ \ \raisebox{-4mm}{
	\begin{tikzpicture}%unitalty
		\draw[red,in=90,out=90,looseness=2] (0,0) to (1,0);
		\node at (.5,.6) {$\red{\bullet}$};
		\node at (0,0) {$\red{\bullet}$};
		\draw[red] (.5,.6) to (.5,1.2);
\end{tikzpicture}}=
\raisebox{-4mm}{
	\begin{tikzpicture}
		\draw[red,in=90,out=90,looseness=2] (0,0) to (1,0);
		\node at (.5,.6) {$\red{\bullet}$};
		\node at (1,0) {$\red{\bullet}$};
		\draw[red] (.5,.6) to (.5,1.2);
\end{tikzpicture}}
=
\raisebox{-2mm}{
	\begin{tikzpicture}
		\draw[red] (0,0) to (0,1.2);
		%\node[red] at (0,0) {$\red{\bullet}$};
\end{tikzpicture}} \ \t{(unitality)} \]

\[\raisebox{-8mm}{
	\begin{tikzpicture}%Frobenius
		\draw[red,in=90,out=90,looseness=2] (0,0) to (1,0);
		\draw[red,in=-90,out=-90,looseness=2] (1,0) to (2,0);
		\node at (.5,.6) {$\red{\bullet}$};
		\node at (1.5,-.6) {$\red{\bullet}$};
		\draw[red] (.5,.6) to (.5,1.2);
		\draw[red] (1.5,-.6) to (1.5,-1.2);
		\draw[red] (0,0) to (0,-.6);
		\draw[red] (2,0) to (2,.6);
\end{tikzpicture}} =
\raisebox{-6mm}{
	\begin{tikzpicture}
		\draw[red,in=90,out=90,looseness=2] (0,0) to (1,0);
		\node at (.5,.6) {$\red{\bullet}$};
		\draw[red] (.5,.6) to (.5,1.2);
		\draw[red,in=-90,out=-90,looseness=2] (0,1.8) to (1,1.8);
		\node at (.5,1.2) {$\red{\bullet}$};
\end{tikzpicture}} =
\raisebox{-8mm}{
	\begin{tikzpicture}
		\draw[red,in=-90,out=-90,looseness=2] (0,0) to (1,0);
		\draw[red,in=90,out=90,looseness=2] (1,0) to (2,0);
		\node at (.5,-.6) {$\red{\bullet}$};
		\node at (1.5,.6) {$\red{\bullet}$};
		\draw[red] (.5,-.6) to (.5,-1.2);
		\draw[red] (1.5,.6) to (1.5,1.2);
		\draw[red] (0,0) to (0,.6);
		\draw[red] (2,0) to (2,-.6);
\end{tikzpicture}} \ \t{(Frobenius condition)} \ \ \ \ \ \ \raisebox{-6mm}{
	\begin{tikzpicture}
		\draw[red,in=90,out=90,looseness=2] (0,0) to (1,0);
		\draw[red,in=-90,out=-90,looseness=2] (0,0) to (1,0);
		\node at (.5,-.6) {$\red{\bullet}$};
		\node at (.5,.6) {$\red{\bullet}$};
		\draw[red] (.5,-.6) to (.5,-1.2);
		\draw[red] (.5,.6) to (.5,1.2);
\end{tikzpicture}} = 
\raisebox{-6mm}{
	\begin{tikzpicture}
		\draw[red] (0,0) to (0,2.4);
\end{tikzpicture}} \ \t{(Separability)} \]

\vspace*{2mm}

In \cite{CPJP}, the notion of \textit{Q-system completion} for C*/W*-2-categories, which is another version of a higher idempotent completion for C*/W*-2-categories in comparison with $ 2 $-categories of separable monads \cite{DR18} and condensation monads \cite{GJF19}, has been introduced, to induce actions of unitary fusion categories on C*-algebras. \textit{Q-system completion} can be thought of as a higher categorical generalization of the `Karoubi envelope' construction from category theory. \textit{Q-system completion} is being used to define the notion of semisimplicity in the context of C*-2-categories \cite{DR18,CPJP}. In recent years, \textit{Q-system completion} has gained considerable interest \cite{CP,G1,G2}.

\vspace*{2mm}

The notion of \textit{Q-system completion} in the context of compact quantum groups has not been much explored. Our goal in this paper is to study the \textit{Q-system completion} of the $ 2 $-category of compact quantum groups.

\vspace*{2mm}

Given a C*/W*-2-category $\mcal C$, which is also locally idempotent complete, its \textit{Q-system completion} is the 2-category $\textbf{QSys}(\mcal C)$ of Q-systems, bimodules and intertwiners in $\mcal C$ . There is a canonical inclusion *-2-functor $\iota_{\mcal C} : \mcal C \hookrightarrow \textbf{QSys}(\mcal C)$ which is always an equivalence on all hom categories. $\mcal C$ is said to be \textit{Q-system complete} if $\iota_{\mcal C}$ is a *-equivalence of *-2-categories. It has been established in \cite{CPJP}, that $\mcal C$ is \textit{Q-system complete} if and only if every Q-system `splits' in $\mcal C$. Our main objective is to study Q-system completion of compact quantum groups.

In \Cref{bimodules}, we form a C*-$ 2 $-category $\mcal G$ whose $ 0 $-cells are compact quantum groups, $ 1 $-cells are unitary bimodules and $ 2 $-cells are intertwiners. Then we prove the first theorem of the paper.

\begin{thm}\label{theorem}
	$\mcal G$ is a locally idempotent complete C*-2-category.
\end{thm}

\Cref{theorem}, allow us to consider the Q-system completion of $\mcal G$. This brings us to the second part of the paper (\Cref{qbl}), where we introduce the notion of \textit{quantum bi-elements}.

\vspace*{2mm}

Quantum pseudo-telepathy \cite{BBT} is a well studied concept in quantum information theory, where two non-communicating parties can use pre-shared entanglement to perform a task classically impossible without communication. Such tasks are usually formulated as games, where isolated players are provided with inputs and must return outputs satisfying some winning conditions. One such game is the \textit{graph isomorphism game} \cite{AMRSSV}, whose instances correspond to pairs of graphs $\Gamma$ and $\Gamma'$, and whose winning classical strategies are exactly graph isomorphisms between $\Gamma$ and $\Gamma'$. Winning quantum strategies are called \textit{quantum isomorphisms}. In \cite{MRV}, the notion of \textit{quantum function} between finite quantum sets has been introduced, which places this game-theoretic approach to quantisation into a broader mathematical context, in particular relating it to more conventional approaches to quantisation.

\vspace*{2mm}

In \cite{MRV}, \textit{a quantum element} of a quantum set (Q-systems in the C*-tensor category of finite dimensional Hilbert spaces) has been defined `quantising' the notion of a copyable state. Our formulation of \textit{quantum bi-elements} can be thought of as two-sided version of a quantum element. In \Cref{qbl}, we explore certain structural properties of \textit{quantum bi-elements}. We prove the following proposition :

\begin{prop}
	Quantum bi-elements give rise to non-unital quantum functions.
\end{prop}

\vspace*{2mm}

Then we turn our attention back to Q-system completion of $\mcal G$. We build a category $\textbf{QBL}$ of quantum bi-elements and intertwiners, and proceed towards our main theorem

\begin{thm}\label{maintheorem}
	The hom categories of $\normalfont\textbf{QSys}(\mcal G)$ form a full subcategory of $\normalfont\textbf{QBL}$.
\end{thm}

There at least two natural questions arising from our work. First, is to identify under what assumptions a quantum function will yield a quantum bi-element. Second, is to check whether the same program can be done in the context of locally compact quantum groups.

\subsection*{Outline of the paper}
In \Cref{prelim}, we discuss the preliminaries related to C*-2-categories, graphical calculus pertaining to it, Q-system completion and compact quantum groups. In \Cref{bimodules}, we form a $ 2 $-category of compact quantum groups and show that is locally idempotent complete. In \Cref{qbl}, we formulate quantum bi-elements and study their structural properties and prove \Cref{maintheorem}.

\subsection*{Acknowledgements}
The author would like to thank Shamindra Ghosh, Makoto Yamashita, Dominic Verdon and Jyotishman Bhowmick for several fruitful discussions.

\section{Preliminaries}\label{prelim}
In this section we will furnish the necessary background and references on $ 2 $-categories, Q-system completion and compact quantum groups that will be needed for our purposes. Towards the end of this section, we discuss the graphical calculus for the C*-tensor category of finite dimensional Hilbert spaces, which will be needed in \Cref{qbl} . 

\subsection{Notations related to 2-categories}\label{graphcalc}
We refer the reader to \cite{JY21} for the basics of 2-categories.\\
Suppose $\mcal C$ is a 2-category and  $a,b \in \mcal C_0$ be two $0$-cells. A $1$-cell from $a  \xrightarrow{X} b$ is denoted by $_aX_b$. Pictorially, a $1$-cell will be denoted by a strand and a $2$-cell will be denoted by a box with strings with passing through it. Suppose we have two $1$-cells $X, Y \in \mcal C_1(a,b)$ and $f \in \mcal C_2(X,Y)$ be a $2$-cell. Then we will denote $f$ as 
\raisebox{-8mm}{
	\begin{tikzpicture}
		\draw (0,.8) to (0,-.8);
		\node[draw,thick,rounded corners,fill=white, minimum width=20] at (0,0) {$f$};
		\node[right] at (0,-.7) {$X$};
		\node[right] at (0,.7) {$Y$};
\end{tikzpicture}}.
We write tensor product $\boxtimes$ of $1$-cells $_aX_b$ and $_bY_c$,  from left to right $ _aX \us{b}\boxtimes Y_c$. Tensor product of two $ 2 $-cells is given by placing the boxes side by side. Composition of $ 2 $-cells is denoted by stacking the boxes one upon another.
$ 1 $-cells which are Q-systems will be denoted by colored strands.
\comments{\red{The notion of C*-2-categories is believed to first appear in \cite{LR}}.} For the basics of C*/W*-2-categories, we refer the reader to \cite{CPJP,GLR85}. For a detailed study about graphical calculus we refer the reader to \cite{HV}.

\subsection{Q-system completion}\

Let $\mcal C$ be a C*-2-category.
\begin{defn} \label{Qsysdefn}
	A Q-system in $\mcal C$ is a $1$-cell $_bQ_b \in \mcal C_1(b,b)$ along with multiplication map $m \in \mcal C_2(Q \boxtimes_b Q, Q)$ and unit map $i \in \mcal C_2(1_b,Q)$, as mentioned in \Cref{introduction}, satisfying the following properties:
\end{defn}
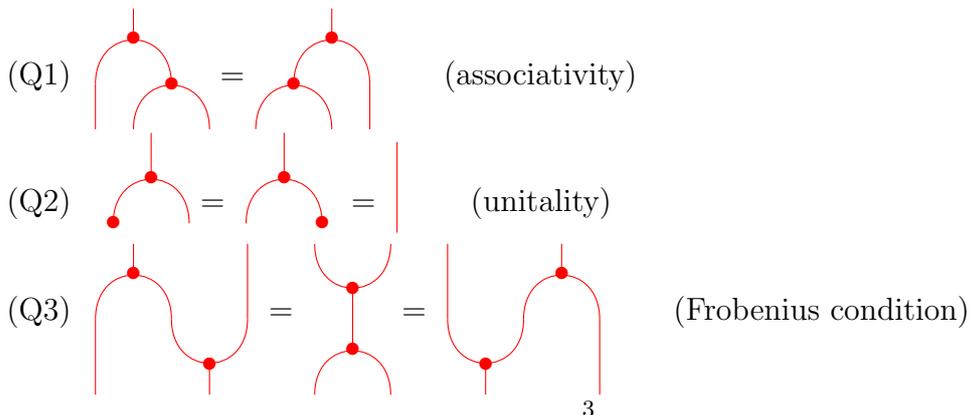
\begin{itemize}
	\item[(Q1)]
	\raisebox{-6mm}{
		\begin{tikzpicture}
			\draw[red,in=90,out=90,looseness=2] (0,0) to (1,0);
			\draw[red,in=90,out=90,looseness=2] (0.5,.6) to (-.5,.6);
			\draw[red] (-.5,.6) to (-.5,0);
			\node at (.5,.6) {$\red{\bullet}$};
			\node at (0,1.2) {$\red{\bullet}$};
			\draw[red] (0,1.2) to (0,1.6);
	\end{tikzpicture}}
	= \raisebox{-6mm}{\begin{tikzpicture}
			\draw[red,in=90,out=90,looseness=2] (0,0) to (1,0);
			\draw[red,in=90,out=90,looseness=2] (.5,.6) to (1.5,.6);
			\node at (.5,.6) {$\red{\bullet}$};
			\draw[red] (1.5,.6) to (1.5,0);
			\node at (1,1.2) {$\red{\bullet}$};
			\draw[red] (1,1.2) to (1,1.6);
	\end{tikzpicture}} \ \ \ \ \ \ \t{(associativity)} 
	
	\item[(Q2)]
	\raisebox{-4mm}{
		\begin{tikzpicture}%unitalty
			\draw[red,in=90,out=90,looseness=2] (0,0) to (1,0);
			\node at (.5,.6) {$\red{\bullet}$};
			\node at (0,0) {$\red{\bullet}$};
			\draw[red] (.5,.6) to (.5,1.2);
	\end{tikzpicture}} = 
	\raisebox{-4mm}{
		\begin{tikzpicture}
			\draw[red,in=90,out=90,looseness=2] (0,0) to (1,0);
			\node at (.5,.6) {$\red{\bullet}$};
			\node at (1,0) {$\red{\bullet}$};
			\draw[red] (.5,.6) to (.5,1.2);
	\end{tikzpicture}}
	= 
	\raisebox{-3mm}{
		\begin{tikzpicture}
			\draw[red] (0,0) to (0,1.2);
	\end{tikzpicture}} \ \ \ \ \ \ \t{(unitality)} 
	
	\item[(Q3)]
	\raisebox{-10mm}{
		\begin{tikzpicture}%Frobenius
			\draw[red,in=90,out=90,looseness=2] (0,0) to (1,0);
			\draw[red,in=-90,out=-90,looseness=2] (1,0) to (2,0);
			\node at (.5,.6) {$\red{\bullet}$};
			\node at (1.5,-.6) {$\red{\bullet}$};
			\draw[red] (.5,.6) to (.5,1);
			\draw[red] (1.5,-.6) to (1.5,-1);
			\draw[red] (0,0) to (0,-1);
			\draw[red] (2,0) to (2,1);
	\end{tikzpicture}} \ =
	\raisebox{-10mm}{
		\begin{tikzpicture}
			\draw[red,in=90,out=90,looseness=2] (0,0) to (1,0);
			\node at (.5,.6) {$\red{\bullet}$};
			\draw[red] (.5,.6) to (.5,1.4);
			\draw[red,in=-90,out=-90,looseness=2] (0,2) to (1,2);
			\node at (.5,1.4) {$\red{\bullet}$};
	\end{tikzpicture}} =
	\raisebox{-10mm}{
		\begin{tikzpicture}
			\draw[red,in=-90,out=-90,looseness=2] (0,0) to (1,0);
			\draw[red,in=90,out=90,looseness=2] (1,0) to (2,0);
			\node at (.5,-.6) {$\red{\bullet}$};
			\node at (1.5,.6) {$\red{\bullet}$};
			\draw[red] (.5,-.6) to (.5,-1);
			\draw[red] (1.5,.6) to (1.5,1);
			\draw[red] (0,0) to (0,1);
			\draw[red] (2,0) to (2,-1);
	\end{tikzpicture}} \ \ \ \ \ \ \t{(Frobenius condition)} 
	
	\item[(Q4)]
	\raisebox{-10mm}{
		\begin{tikzpicture}
			\draw[red,in=90,out=90,looseness=2] (0,0) to (1,0);
			\draw[red,in=-90,out=-90,looseness=2] (0,0) to (1,0);
			\node at (.5,-.6) {$\red{\bullet}$};
			\node at (.5,.6) {$\red{\bullet}$};
			\draw[red] (.5,-.6) to (.5,-1.2);
			\draw[red] (.5,.6) to (.5,1.2);
	\end{tikzpicture}} = 
	\raisebox{-10mm}{
		\begin{tikzpicture}
			\draw[red] (0,0) to (0,2.4);
	\end{tikzpicture}} \ \ \ \ \ \ \t{(Separability)}
	
\end{itemize}
\begin{rem}\label{Qdual}
	A Q-system $Q$ is a self-dual 1-cell with $ev_Q \coloneqq 
	\raisebox{-4mm}{
		\begin{tikzpicture}[scale=.8]
			\draw[red,in=-90,out=-90,looseness=2] (-0.5,0.5) to (-1.5,0.5);
			\node at (-1,-.1) {${\color{red}\bullet}$};
			\draw[red] (-1,-.1) to (-1,-.6);
			\node[left,scale=0.7] at (-1,-.4) {$Q$};
			\node[left,scale=0.7] at (-1.6,0.4) {$Q$};
			\node[right,scale=0.7] at (-.5,.4) {$Q$};
			\node at (-1,-.6) {${\color{red}\bullet}$};
	\end{tikzpicture}}
	$ and $coev_Q \coloneqq \raisebox{-4mm}{
		\begin{tikzpicture}[scale=.8]
			\draw[red,in=90,out=90,looseness=2] (-0.5,0.5) to (-1.5,0.5);
			\node at (-1,1.1) {${\color{red}\bullet}$};
			\draw[red] (-1,1.1) to (-1,1.6);
			\node[left,scale=0.7] at (-1,1.4) {$Q$};
			\node[left,scale=0.7] at (-1.6,0.6) {$Q$};
			\node[right,scale=0.7] at (-.5,.6) {$Q$};
			\node at (-1,1.6) {${\color{red}\bullet}$};
	\end{tikzpicture}}$.
\end{rem}
\begin{defn}
	Suppose $\mcal C$ is a C*-2-category and $_aX_b \in \mcal C_1(a,b)$. A \textit{unitarily separable left dual} for $_a X_b$ is a dual $\left( _b \ol X_a, ev_X, coev_X \right)$ such that $ev_X \circ ev_X^* = \t{id}_{1_b}$ (cf. \cite[Example 3.9]{CPJP}).
\end{defn}
Given a unitarily separable left dual for $_aX_b \in \mcal C_1(a,b)$, $_a X \us{b} \boxtimes \ol X_a \in \mcal C_1(b,b)$ is a Q-system with multiplication map $m \coloneqq \t{id}_X \boxtimes ev_X \boxtimes \t{id}_{\ol X}$ and unit map $i \coloneqq coev_X$.

Given a Q-system $Q \in \mcal C_1(b,b)$, if it is of the above form,  then we say that the Q-system $Q$ `splits'.

\begin{defn}
	Suppose $\mcal C$ is a C*-2-category. $\mcal C$ is said to be \textit{locally idempotent complete} if, given any projection $p \in \t{End}(X)$ for a $ 1 $-cell $X \in \mcal C_1(a,b)$, there is another $ 1 $-cell $Y \in \mcal C_1(a,b)$ and an isometry $v \in \mcal C_2(Y,X)$ such that $v v^* = p$ .  
\end{defn}

\begin{defn}\cite{CPJP}
	Suppose $\mcal C$ is a locally idempotent complete C*-2-category, and $P \in \mcal C_1(a,a)$ and $Q \in \mcal C_1(b,b)$ be Q-systems in $\mcal C$. A $Q-P$ bimodule is a triple $\left( X, \lambda_X, \rho_X \right)$ consisting of $_aX_b \in \mcal C_1(a,b)$, $\lambda_X \in \mcal C_2 \left(Q \boxtimes X, X\right)$ and $\rho_X \in \mcal C_2(X \boxtimes P, X)$ satisfying certain properties. We represent $X, P, Q$ graphically by purple, red and magenta strands respectively. We denote $\lambda_X, \rho_X, \lambda_X^*, \rho_X^*$ as follows:
	\[\lambda_X \ = \ \raisebox{-2mm}{\begin{tikzpicture}
			\draw[purple] (0,0) to (0,.6);%X
			\draw[red][in=120,out=90,looseness=1] (-.5,0) to (0,.3);%P
			\node[scale=.8] at (0,.3) {$\color{purple}{\bullet}$};
	\end{tikzpicture}} \ \ \ \ 
	\rho_X \ = \ \raisebox{-2mm}{\begin{tikzpicture}
			\draw[purple] (0,0) to (0,.6);%X
			\draw[magenta][in=60,out=90,looseness=1] (.5,0) to (0,.3);%Q
			\node[scale=.8] at (0,.3) {$\color{purple}{\bullet}$};
	\end{tikzpicture}} \ \ \ \
	\lambda_X^* \ = \ \raisebox{-2mm}{\begin{tikzpicture}[rotate=180]
			\draw[purple] (0,0) to (0,.6);%X
			\draw[red][in=60,out=90,looseness=1] (.5,0) to (0,.3);%Q
			\node[scale=.8] at (0,.3) {$\color{purple}{\bullet}$};
	\end{tikzpicture}} \ \ \ \
	\rho_X^* \ = \ \raisebox{-2mm}{\begin{tikzpicture}[rotate=180]
			\draw[purple] (0,0) to (0,.6);%X
			\draw[magenta][in=120,out=90,looseness=1] (-.5,0) to (0,.3);%P
			\node[scale=.8] at (0,.3) {$\color{purple}{\bullet}$};
	\end{tikzpicture}}  \]
	The bimodule axioms are as follows:
	\begin{itemize}
		\item[(B1)] \raisebox{-4mm}{\begin{tikzpicture}
				\draw[purple] (0,0) to (0,1.2);
				\draw[red][in=120,out=90,looseness=1] (-.5,0) to (0,.5);
				\draw[red][in=120,out=90,looseness=1] (-.7,0) to (0,.7);
				\node[scale=.8] at (0,.52) {$\color{purple}{\bullet}$};
				\node[scale=.8] at (0,.72) {$\color{purple}{\bullet}$};
				\node at (.6,.6) {$=$};
				\draw[red,in=90,out=90,looseness=2] (1,0) to (1.5,0);
				\node[scale=.7] at (1.25,.27) {$\red{\bullet}$};
				\draw[red,in=120,out=90,looseness=1] (1.25,.27) to (1.75,.57);
				\draw[purple] (1.75,0) to (1.75,1.2);
				\node[scale=.8] at (1.75,.59) {$\color{purple}{\bullet}$};
		\end{tikzpicture}} \hspace*{2mm}, \hspace*{2mm}
		\raisebox{-4mm}{\begin{tikzpicture}
				\draw[purple] (0,0) to (0,1.2);
				\draw[magenta][in=60,out=90,looseness=1] (.5,0) to (0,.5);
				\draw[magenta][in=60,out=90,looseness=1] (.7,0) to (0,.7);
				\node[scale=.8] at (0,.52) {$\color{purple}{\bullet}$};
				\node[scale=.8] at (0,.72) {$\color{purple}{\bullet}$};
				\node at (1,.6) {$=$};
				\draw[purple] (1.3,0) to (1.3,1.2);
				\draw[magenta,in=90,out=90,looseness=2] (1.55,0) to (2.05,0);
				\draw[magenta,in=60,out=90,looseness=1] (1.8,.27) to (1.3,.57);
				\node[scale=.7] at (1.8,.27) {$\color{magenta}{\bullet}$};
				\node[scale=.8] at (1.3,.57) {$\color{purple}{\bullet}$};
		\end{tikzpicture}} \hspace*{2mm} and \hspace*{2mm}
		\raisebox{-4mm}{\begin{tikzpicture}
				\draw[purple] (0,0) to (0,1);%X
				\draw[red][in=120,out=90,looseness=1] (-.5,0) to (0,.6);%P
				\draw[magenta][in=60,out=90,looseness=1] (.5,0) to (0,.3);%Q
				\node[scale=.8] at (0,.6) {$\color{purple}{\bullet}$};
				\node[scale=.8] at (0,.3) {$\color{purple}{\bullet}$};
				\node at (.7,.6) {$=$};
				\draw[purple] (1.5,0) to (1.5,1);%X
				\draw[red][in=120,out=90,looseness=1] (1,0) to (1.5,.3);%P
				\draw[magenta][in=60,out=90,looseness=1] (2,0) to (1.5,.6);%Q
				\node[scale=.8] at (1.5,.6) {$\color{purple}{\bullet}$};
				\node[scale=.8] at (1.5,.3) {$\color{purple}{\bullet}$};
		\end{tikzpicture}} \hspace*{2mm} (associativity)
		\item[(B2)] \raisebox{-4mm}{\begin{tikzpicture}
				\draw[purple] (0,0) to (0,1);%X
				\draw[red][in=120,out=90,looseness=1] (-.5,.2) to (0,.6);%P
				\node[scale=.7] at (-.5,.2) {$\red{\bullet}$};
				\node[scale=.8] at (0,.6) {$\color{purple}{\bullet}$};
				\node at (.4,.6) {$=$};
				\draw[purple] (.8,0) to (.8,1);
				\node at (1.2,.6) {$=$};
				\draw[purple] (1.6,0) to (1.6,1);
				\draw[magenta][in=60,out=90,looseness=1] (2.1,.2) to (1.6,.6);%Q
				\node[scale=.7] at (2.1,.2) {$\color{magenta}{\bullet}$};
				\node[scale=.8] at (1.6,.6) {$\color{purple}{\bullet}$};
		\end{tikzpicture}} \hspace*{2mm} (unitality)
		
		\item[(B3)] \raisebox{-6mm}{\begin{tikzpicture}
				\draw[purple] (0,0) to (0,1.2);%X
				\draw[red][in=-120,out=-90,looseness=1] (-.3,.5) to (0,.3);
				\draw[red,in=90,out=90,looseness=2] (-.7,.5) to (-.3,.5);
				\draw[red] (-.7,.5) to (-.7,0);
				\draw[red] (-.5,.7) to (-.5,1);
				\node[scale=.7] at (-.5,.72) {$\red{\bullet}$};
				\node[scale=.8] at (0,.3) {$\color{purple}{\bullet}$};
				\node at (.4,.6) {$=$};
				\draw[purple] (1.2,0) to (1.2,1.2);
				\draw[red][in=-120,out=-90,looseness=1] (.8,1.2) to (1.2,.8);
				\draw[red][in=120,out=90,looseness=1] (.8,0) to (1.2,.4);
				\node[scale=.8] at (1.2,.8) {$\color{purple}{\bullet}$};
				\node[scale=.8] at (1.2,.4) {$\color{purple}{\bullet}$};
				\node at (1.6,.6) {$=$};
				\draw[purple] (2.8,0) to (2.8,1.2);
				\draw[red][in=120,out=90,looseness=1] (2.4,.4) to (2.8,.8);
				\draw[red,in=-90,out=-90,looseness=2] (2.4,.4) to (2,.4);
				\draw[red] (2,.4) to (2,1);
				\draw[red] (2.2,.2) to (2.2,-.1);
				\node[scale=.7] at (2.2,.15) {$\red{\bullet}$};
				\node[scale=.8] at (2.8,.8) {$\color{purple}{\bullet}$};
		\end{tikzpicture}} \hspace*{2mm} and \hspace*{2mm}
		\raisebox{-6mm}{\begin{tikzpicture}
				\draw[purple] (0,0) to (0,1.2);
				\draw[magenta][in=-60,out=-90,looseness=1] (.3,.5) to (0,.3);
				\draw[magenta][in=90,out=90,looseness=2] (.3,.5) to (.7,.5);
				\draw[magenta] (.5,.7) to (.5,1); 
				\draw[magenta] (.7,.5) to (.7,0);
				\node[scale=.8] at (0,.3) {$\color{purple}{\bullet}$};
				\node[scale=.7] at (.5,.7) {$\color{magenta}{\bullet}$};
				\node at (1.1,.6) {$=$};
				\draw[purple] (1.5,0) to (1.5,1.2);
				\draw[magenta][in=-60,out=-90,looseness=1] (1.9,1.2) to (1.5,.8);
				\draw[magenta][in=60,out=90,looseness=1] (1.9,0) to (1.5,.4);
				\node[scale=.8] at (1.5,.8) {$\color{purple}{\bullet}$};
				\node[scale=.8] at (1.5,.4) {$\color{purple}{\bullet}$};
				\node at (2.3,.6) {$=$};
				\draw[purple] (2.7,0) to (2.7,1.2);
				\draw[magenta][in=60,out=90,looseness=1] (3,.4) to (2.7,.7);
				\draw[magenta,in=-90,out=-90,looseness=2] (3,.4) to (3.4,.4);
				\draw[magenta] (3.4,.4) to (3.4,1);
				\draw[magenta] (3.2,.2) to (3.2,0);
				\node[scale=.7] at (3.2,.15) {$\color{magenta}{\bullet}$};
				\node[scale=.8] at (2.7,.7) {$\color{purple}{\bullet}$}; 
		\end{tikzpicture}} \hspace*{2mm} (frobenius condition)
		
		\item[(B4)] \raisebox{-6mm}{\begin{tikzpicture}
				\draw[purple] (0,0) to (0,1.2);%X
				\draw[red][in=-120,out=-90,looseness=1] (-.3,.5) to (0,.3);%P
				\draw[red][in=120,out=90,looseness=1] (-.3,.5) to (0,.8);
				\node[scale=.8] at (0,.3) {$\color{purple}{\bullet}$};
				\node[scale=.8] at (0,.8) {$\color{purple}{\bullet}$};
				\node at (.4,.6) {$=$};
				\draw[purple] (.8,0) to (.8,1.2);
				\node at (1.2,.6) {$=$};
				\draw[purple] (1.6,0) to (1.6,1.2);
				\draw[magenta][in=-60,out=-90,looseness=1] (1.9,.5) to (1.6,.3);
				\draw[magenta][in=60,out=90,looseness=1] (1.9,.5) to (1.6,.8);
				\node[scale=.8] at (1.6,.3) {$\color{purple}{\bullet}$};
				\node[scale=.8] at (1.6,.8) {$\color{purple}{\bullet}$};
		\end{tikzpicture}} \hspace*{2mm} (separability)
	\end{itemize}
\end{defn}

\begin{defn}
	Given a C*/W*-2-category $\mcal C$, it's \textit{Q-system completion} is the C*/W*-2-category $\textbf{QSys}(\mcal C)$ whose :
	\begin{itemize}
		\item[(1)] 0-cells are Q-systems $(Q,m,i) \in \mcal C_1(b,b)$.
		\item[(2)] 1-cells between Q-systems $P \in \mcal C_1(a,a)$ and $Q \in \mcal C_1(b,b)$ are $Q-P$ bimodules $(X, \lambda_X, \rho_X)$.
		\item[(3)] 2-cells are bimodule intertwiners. Given Q-systems $_aP_a$, $_bQ_b$ and $Q-P$ bimodules $_bX_a$, $_bY_a$, $\textbf{QSys}(\mcal C)_2(_QX_P,\ _QY_P)$ is the set of $f \in \mcal C_2(X,Y)$ such that:
		\[\raisebox{-4mm}{\begin{tikzpicture}
				\draw[purple] (0,0) to (0,1.6);%X
				\draw[red][in=120,out=90,looseness=1] (-.5,0) to (0,.3);%P
				\node[scale=.8] at (0,.3) {$\color{purple}{\bullet}$};
				\node[draw,thick,rounded corners,fill=white] at (0,.9) {$f$};
				\node at (.6,.8) {$=$};
				\draw[purple] (1.5,0) to (1.5,1.6);%X
				\draw[red][in=120,out=90,looseness=1] (1,.8) to (1.5,1.1);%P
				\node[scale=.8] at (1.5,1.1) {$\color{purple}{\bullet}$};
				\node[draw,thick,rounded corners,fill=white] at (1.5,.6) {$f$};
				\draw[red] (1,.8) to (1,0);
		\end{tikzpicture}} \ \ \ \ \t{and}\ \ \ \
		\raisebox{-4mm}{\begin{tikzpicture}
				\draw[purple] (0,0) to (0,1.6);%X
				\draw[magenta][in=60,out=90,looseness=1] (.5,0) to (0,.3);%Q
				\node[scale=.8] at (0,.3) {$\color{purple}{\bullet}$};
				\node[draw,thick,rounded corners,fill=white] at (0,.9) {$f$};
				\node at (.6,.8) {$=$};
				\draw[purple] (1.2,0) to (1.2,1.6);
				\draw[magenta][in=60,out=90,looseness=1] (1.7,.8) to (1.2,1.1);
				\node[scale=.8] at (1.2,1.1) {$\color{purple}{\bullet}$};
				\node[draw,thick,rounded corners,fill=white] at (1.2,.6) {$f$};
				\draw[magenta] (1.7,.8) to (1.7,0);
		\end{tikzpicture}} \]
	\end{itemize}
	We refer the reader to \cite{CPJP} for full details that $\textbf{QSys}(\mcal C)$ is a C*-2-category.
\end{defn}

Let us recall the description of canonical inclusion functor *-2-functor $\iota_{\mcal C} : \mcal C \hookrightarrow \textbf{QSys}\left(\mcal C\right)$ already given in \cite[Construction 3.24]{CPJP}.
\begin{itemize}
	\item[(I1)] For $c \in \mcal C_0$, we map $c$ to $1_c \in \mcal C_1(c,c)$ with its obvious Q-system structure as the tensor unit of $\mcal C_1(c,c)$.
	
	\item[(I2)] For a 1-cell  $_a X_b \in \mcal C_1(a,b)$, $X$ itself is a unital Frobenius $1_a$-$1_b$ bimodule object, so we map $X$ to itself.
	
	\item[(I3)] For a 2-cell $f \in \mcal C_2(X,Y)$, we see that $f$ is $1_a - 1_b$ bimodular, so we map $f$ to itself. 
\end{itemize}
For further details about the canonical inclusion functor *-2-functor $\iota_{\mcal C} : \mcal C \hookrightarrow \textbf{QSys}\left(\mcal C\right)$ we refer the reader to \cite{CPJP}. 
\begin{defn}
	A 2-category $\mcal C$ is said to be Q-system complete if the canonical inclusion functor *-2-functor $\iota_{\mcal C} : \mcal C \hookrightarrow \textbf{QSys}\left(\mcal C\right)$ is a *-2-equivalence.
\end{defn}
In \cite{CPJP}, Q-system completion of a C*-2-category $\mcal C$  has been characterised in terms of Q-systems which `splits'.

\begin{thm}\cite[Theorem 3.36]{CPJP}.
	A C*/W*-2-category is said to be Q-system complete if and only if every Q-system $Q \in \mcal C_1(b,b)$ splits, that is, there is an object $c \in \mcal C_0$ and a dualizable 1-cell $X \in \mcal C_1(c,b)$ which admits a unitary separable dual $\left(\ol{X}, ev_X, coev_X \right)$ such that $\left(Q,m,i\right)$ is isomorphic to $_b X \us{c}\boxtimes \ol{X}_b$ as Q-systems.
\end{thm}

\subsection{Compact quantum groups}\

\begin{defn}
	A compact quantum group is a pair $\left(\mcal A, \Delta\right)$, where $\mcal A$ is a unital C*-algebra and $\Delta : A \to A \otimes A$ is a unital *-homomorphism, called co-multiplication, such that :
	\begin{itemize}
		\item [(i)] $\left(\Delta \otimes 1\right)\Delta = \left(1 \otimes \Delta\right)\Delta $
		\item[(ii)]  the spaces $\left(A \otimes 1\right)\Delta(A)$ and $\left(1 \otimes A\right)\Delta(A)$ are dense in $A \otimes A$.
	\end{itemize}
\end{defn}

\begin{defn}\cite{NT}
	A \textit{right unitary co-representation} of a compact quantum group $\left(\mathbb{G}, \Delta\right)$ on a finite dimensional Hilbert space $H$ is an unitary $U \in B(H) \otimes C(\mathbb{G})$ such that
	  \[\left(1 \otimes \Delta\right)(U) = U_{(12)} U_{(13)} \in B(H) \otimes C(\mathbb{G}) \otimes C(\mathbb{G}) \]
\end{defn}
Similarly, we call an unitary $V \in C(\mathbb{G}) \otimes B(H)$ to be a \textit{left unitary co-representation} if, $\left(\Delta \otimes 1 \right)(V) = V_{(13)} V_{(23)} $

Suppose $\mathbb{G}$ and $\mathbb{H}$ be two compact quantum groups. For a finite dimensional Hilbert space $H$, $H \otimes C(\mathbb{G})$ has a $C(\mathbb{G})$-valued inner product given by,
\[\lab \xi \otimes x , \eta \otimes y \rab_{C(\mathbb{G})} = \lab \xi, \eta \rab x^* y  \ \ \ \ \t{for all} \ \ \xi, \eta \in H \ \ \t{and} \ \ x , y \in C(\mathbb{G}). \]
Similarly, $C(\mathbb{H}) \otimes H$ has a $C(\mathbb{H})$-valued inner product.

\begin{defn}
	A finite dimensional Hilbert space $ V $ is said to be a \textit{right unitary} $\mathbb{G}$- \textit{module} if there exist a linear map $\alpha_{V, \mathbb{G}} : V \to V \otimes C(\mathbb{G}) $ satisfying the following :
	\begin{itemize}
		%\item [(i)] $\left( \Delta \otimes 1_V \right) \circ \alpha_{\mathbb{G,V}} $ = $\left(1_{\mathbb{G}} \otimes \alpha_{\mathbb{G}, V}\right) \circ \alpha_{\mathbb{G}, V} $ \ \ \ \ (left co-action)
		\item [(i)] $\left(1_V \otimes \Delta \right) \circ \alpha_{V, \mathbb{G}}$ = $\left(\alpha_{V, \mathbb{G}} \otimes 1_{\mathbb{G}}\right) \circ \alpha_{V, \mathbb{G}}$ \ \  \ \ (right co-action)
		%\item [(iii)] $\left(1_{\mathbb{G}} \otimes \alpha_{V, \mathbb{H}} \right) \circ \alpha_{\mathbb{G}, V}$ = $\left(\alpha_{\mathbb{G}, V} \otimes 1_{\mathbb{H}}\right) \circ \alpha_{V, \mathbb{H}}$ \ \ \ \ (compatibility of left and right co-actions)
		\item [(ii)] $\lab \alpha_{V, \mathbb{G}}(\xi), \alpha_{V, \mathbb{G}}(\eta) \rab_{C(\mathbb{G})} = \lab \xi, \eta \rab 1_{\mathbb{G}}$ 
	\end{itemize}
\end{defn}
Similarly, we call a finite dimensional Hilbert space $V$ to be a \textit{left unitary $\mathbb{G}$-module} if there exists a linear map $\alpha_{\mathbb{G}, V} : V \to C(\mathbb{G}) \otimes V $ satisfying the following :
\begin{itemize}
	\item [(i)] $ \left( \Delta \otimes 1_V \right) \circ \alpha_{\mathbb{G,V}} $ = $\left(1_{\mathbb{G}} \otimes \alpha_{\mathbb{G}, V}\right) \circ \alpha_{\mathbb{G}, V} $ \ \ \ \ (left co-action)
	
	\item [(ii)] $\lab \alpha_{\mathbb{G}, V}(\xi), \alpha_{\mathbb{G}, V}(\eta) \rab_{C(\mathbb{G})} = \lab \xi, \eta \rab 1_{\mathbb{G}}$
\end{itemize}

\begin{prop}\cite[Lemma 1.7]{KDC}\label{bimcorr} 
	There is a one-to-one correspodence between right (resp. left) unitary co-representations of $\mathbb{G}$ and right (resp. left) unitary $\mathbb{G}$-modules.
\end{prop}

\subsection{Graphical calculus in Hilbert spaces}\

The graphical calculus (also known as string diagram calculus) has already been discussed in \Cref{graphcalc}, in the context of $ 2 $-categories. In this subsection, for the convenience of the reader, we unpack and discuss the graphical calculus in the context of finite dimensional Hilbert spaces. This will be used in \Cref{qbl}.

Throughout the paper, we will denote by $\textbf{Hilb}$, the category of finite dimensional Hilbert spaces with linear maps between them as morphisms. Suppose $H$ and $K$ be two Hilbert spaces and $f : H \to K$ be a linear map. We denote $f$ and the identity maps $1_H$ and $1_K$ graphically as follows : $$f = \raisebox{-6mm}{\begin{tikzpicture}
	\draw (0,-.7) to (0,.7);
	\node[draw,thick,rounded corners,fill=white] at (0,0) {$f$};
	\node[scale=.8] at (.2,-.6) {$H$};
	\node[scale=.8] at (.2,.6) {$K$};
\end{tikzpicture}} \ \ \ \ \ \  1_H = \raisebox{-4mm}{\begin{tikzpicture}
\draw (0,0) to (0,1);
\node[scale=.8] at (.2,.2) {$H$};
\end{tikzpicture}} \ \ \ \  \ \ 1_K = \raisebox{-4mm}{\begin{tikzpicture}
\draw (0,0) to (0,1);
\node[scale=.8] at (.2,.2) {$K$};
\end{tikzpicture}}$$ 
When the Hilbert spaces are clear from the context we will omit the labelling of the strands.

Suppose $f_1 : H \otimes K \to H' \otimes K'$, $f_2 : H \to H'$ and $f_3 : K \to K' $ be linear maps. Then $f_1$ and $f_2 \otimes f_3$ will be denoted as :
\[f_1 \ = \ \raisebox{-6mm}{\begin{tikzpicture}
	\draw (-.3,-.7) to (-.3,.7);
	\draw (.3,-.7) to (.3,.7);
	\node[draw,thick,rounded corners,fill=white,minimum width=30] at (0,0) {$f$};
\end{tikzpicture}} \ \ \ \ \ \ \ \ f_2 \otimes f_3 \ = \ \raisebox{-6mm}{\begin{tikzpicture}
\draw (0,-.7) to (0,.7);
\draw (.8,-.7) to (.8,.7);
\node[draw,thick,rounded corners,fill=white] at (0,0) {$f_2$};
\node[draw,thick,rounded corners,fill=white] at (.8,0) {$f_3$};
\end{tikzpicture}}  \] 
Composition of linear maps $f : H \to K$ and $g :K \to L$ is given by,
\[g \circ f \ = \ \raisebox{-8mm}{\begin{tikzpicture}
	\draw (0,-.6) to (0,1.6);
	\node[draw,thick,rounded corners,fill=white] at (0,1) {$g$};
	\node[draw,thick,rounded corners,fill=white] at (0,0) {$f$};
\end{tikzpicture}} \]

\section{Bimodules of compact quantum groups}\label{bimodules}
In this section, we will build the $ 2 $-category $\mcal G$ whose $ 0 $-cells are compact quantum groups, $ 1 $-cells are unitary bimodules and $ 2 $-cells are intertwiners. Our goal is to study Q-system completion of $\mcal G$. A necessary ingredient that allows one to consider Q-system completion is idempotent completeness. In \Cref{idemthm}, we prove idempotent completeness of $\mcal G$. Finally, in \Cref{qsysrem}, we show that $\mcal G$ is not Q-system complete.

\begin{defn}\label{bimoduledefn}
	A finite dimensional Hilbert space $ V $ is said to be a \textit{unitary} $\mathbb{G}$-$\mathbb{H}$ \textit{bimodule} if there exist linear maps $\alpha_{\mathbb{G}, V} : V \to C(\mathbb{G}) \otimes V $ and $\alpha_{V, \mathbb{H}} : V \to V \otimes C(\mathbb{H}) $ satisfying the following :
	\begin{itemize}
		\item [(i)] $\left( \Delta \otimes 1_V \right) \circ \alpha_{\mathbb{G,V}} $ = $\left(1_{\mathbb{G}} \otimes \alpha_{\mathbb{G}, V}\right) \circ \alpha_{\mathbb{G}, V} $ \ \ \ \ (left co-action)
		\item [(ii)] $\left(1_V \otimes \Delta \right) \circ \alpha_{V, \mathbb{H}}$ = $\left(\alpha_{V, \mathbb{H}} \otimes 1_{\mathbb{H}}\right) \circ \alpha_{V, \mathbb{H}}$ \ \  \ \ (right co-action)
		\item [(iii)] $\left(1_{\mathbb{G}} \otimes \alpha_{V, \mathbb{H}} \right) \circ \alpha_{\mathbb{G}, V}$ = $\left(\alpha_{\mathbb{G}, V} \otimes 1_{\mathbb{H}}\right) \circ \alpha_{V, \mathbb{H}}$ \ \ \ \ (compatibility of left and right co-actions)
		\item [(iv)] $\lab \alpha_{\mathbb{G}, V}(\xi), \alpha_{\mathbb{G}, V}(\eta) \rab_{C(\mathbb{G})} = \lab \xi, \eta \rab 1_{\mathbb{G}}$ and $\lab \alpha_{V, \mathbb{H}}(\xi), \alpha_{V, \mathbb{H}}(\eta) \rab_{C(\mathbb{H})} = \lab \xi, \eta \rab 1_{\mathbb{H}}$ 
	\end{itemize}
\end{defn}

\begin{ex}\label{example}
	We say that a left (resp. right) $\mathbb{G}$ co-action on a Hilbert space $ V $ is trivial if $\alpha_{\mathbb{G}, V} (\xi) = 1_{\mathbb{G}} \otimes \xi $ (resp. $\alpha_{V, \mathbb{G}} (\xi) = \xi \otimes 1_{\mathbb{G}}$), for all $\xi \in V $. Thus, any finite dimensional Hilbert space $H$ equipped with trivial left $\mathbb{G}$ co-action and a right $\mathbb{H}$ co-action becomes a unitary $\mathbb{G}$-$\mathbb{H}$ bimodule.
\end{ex}
\begin{ex}
	Suppose $U \in B(H) \otimes C(\mathbb{G})$ be a unitary co-representation. Then by \Cref{bimcorr}, there is a unitary right co-action of $\mathbb{G}$ on $H$. Taking a trivial left co-action of any compact quantum group $\mathbb{H}$ on $H$ yields a unitary $\mathbb{H}$-$\mathbb{G}$ bimodule $H$.
\end{ex}

\begin{lem}
	The direct sum $V \oplus W$ of two unitary $\mathbb{G}$-$\mathbb{H}$ bimodules $ V $ and $ W $ is again a unitary $\mathbb{G}$-$\mathbb{H}$ bimodule.
\end{lem}
\begin{proof}
	Suppose $u : V \to V \oplus W$ and $v : W \to V \oplus W$ be the orthogonal isometries. Consider the linear maps,
	\[\alpha_{\mathbb{G}, V \oplus W} \coloneqq \left(1 \otimes u\right) \alpha_{\mathbb{G}, V} u^* + \left(1 \otimes v\right) \alpha_{\mathbb{G}, W} v^*  \]
	\[\alpha_{V \oplus W, \mathbb{H}} \coloneqq \left(u \otimes 1\right) \alpha_{V,\mathbb{H}} u^* + \left(v \otimes 1\right) \alpha_{W, \mathbb{H}} v^*  \]
	It is a straight forward verification that $\alpha_{\mathbb{G}, V \oplus W}$ and $\alpha_{V \oplus W, \mathbb{H}}$ satisfy the conditions of \Cref{bimoduledefn}. Thus, $V \oplus W$ becomes a unitary $\mathbb{G}$-$\mathbb{H}$ bimodule.
\end{proof}

In the next result we provide a monoidal structure of unitary bimodules.
\begin{prop}\label{bimtensorproduct}
	Suppose $ V $ is a unitary $\mathbb{G}$-$\mathbb{H}$ bimodule and $ W $ is a unitary $\mathbb{H}$-$\mathbb{K}$ bimodule. Then $ V \otimes W $ is a unitary $\mathbb{G}$-$\mathbb{K}$ bimodule.
\end{prop}
\begin{proof}
	We first define the left and right co-actions on $ V \otimes W$. Define 
	\[\alpha_{\mathbb{G}, V \otimes W} \coloneqq \alpha_{\mathbb{G}, V} \otimes 1_W  \ \ \ \ \t{and} \ \ \ \ \alpha_{V \otimes W, \mathbb{K}} \coloneqq 1_V \otimes \alpha_{  W, \mathbb{K}}\]
	\begin{align*}
		\left(\Delta \otimes 1_{V \otimes W} \right) \circ \alpha_{\mathbb{G}, V \otimes W} &= \left(\Delta \otimes 1_V \right) \circ \alpha_{\mathbb{G}, V } \otimes W \\
		&= \left(1_{\mathbb{G}} \otimes \alpha_{\mathbb{G}, V }\right) \circ \alpha_{\mathbb{G}, V } \otimes 1_W  \ \ (\t{Since $V$ has a left $\mathbb{G}$ co-action})
	\end{align*}
We also have that,
\begin{align*}
	\left(1_{\mathbb{G}} \otimes \alpha_{\mathbb{G}, V \otimes W}\right) \circ \alpha_{\mathbb{G}, V \otimes W} &= \left(1_{\mathbb{G}} \otimes \alpha_{\mathbb{G}, V } \otimes 1_W \right) \circ \left(\alpha_{\mathbb{G}, V } \otimes 1_W \right) \\
	&= \left(1_{\mathbb{G}} \otimes \alpha_{\mathbb{G}, V }\right) \circ \alpha_{\mathbb{G}, V } \otimes 1_W
\end{align*}
Therefore, $\left(1_{\mathbb{G}} \otimes \alpha_{\mathbb{G}, V \otimes W}\right) \circ \alpha_{\mathbb{G}, V \otimes W} = \left(1_{\mathbb{G}} \otimes \alpha_{\mathbb{G}, V \otimes W}\right) \circ \alpha_{\mathbb{G}, V \otimes W}$. Thus, $\alpha_{\mathbb{G}, V \otimes W}$ satisfy condition (i) of \Cref{bimoduledefn}. In a similar way, it can be shown that $\alpha_{V \otimes W, \mathbb{K}}$ satisfy condition (ii) of \Cref{bimoduledefn}. We now prove the compatibility of the left and right co-actions. We have,
\[\left(\alpha_{\mathbb{G}, V \otimes W} \otimes 1_{\mathbb{K}}\right) \circ \alpha_{V \otimes W, \mathbb{K}} = \left(\alpha_{\mathbb{G}, V } \otimes 1_W \otimes 1_{\mathbb{K}}\right) \circ \left(1_V \otimes \alpha_{  W, \mathbb{K}}\right) = \alpha_{\mathbb{G}, V } \otimes \alpha_{  W, \mathbb{K}} \] 
\[\left(1_{\mathbb{G}} \otimes \alpha_{V \otimes W, \mathbb{K}}\right) \circ \alpha_{\mathbb{G}, V \otimes W} = \left(1_{\mathbb{G}} \otimes 1_V \otimes \alpha_{  W, \mathbb{K}} \right) \circ \left(\alpha_{\mathbb{G}, V } \otimes 1_W \right) = \alpha_{\mathbb{G}, V } \otimes \alpha_{  W, \mathbb{K}} \]
Thus, we get $\left(\alpha_{\mathbb{G}, V \otimes W} \otimes 1_{\mathbb{K}}\right) \circ \alpha_{V \otimes W, \mathbb{K}} = \left(1_{\mathbb{G}} \otimes \alpha_{V \otimes W, \mathbb{K}}\right) \circ \alpha_{\mathbb{G}, V \otimes W}$. Hence, we have condition (iii) of \Cref{bimoduledefn} for $\alpha_{\mathbb{G}, V \otimes W}$ and $\alpha_{V \otimes W, \mathbb{K}}$. Finally, it is trivial to verify condition (iv) for $\alpha_{\mathbb{G}, V \otimes W}$ and $\alpha_{V \otimes W, \mathbb{K}}$.  
\end{proof}
\begin{defn}
	Suppose $ V $ and $ W $ be two unitary $\mathbb{G}$-$\mathbb{H}$ bimodules. A linear map $T : V \to W$ is said to be a $\mathbb{G}$-$\mathbb{H}$ \textit{intertwiner} if the following holds:
	\begin{itemize}
		\item [(i)] $\left(1_{\mathbb{G}} \otimes T\right) \circ \alpha_{\mathbb{G}, V } = \alpha_{\mathbb{G}, W} \circ T $
		\item [(ii)] $\left(T \otimes 1_{\mathbb{H}} \right) \circ \alpha_{V , \mathbb{H}} = \alpha_{  W, \mathbb{H}} \circ T $
		\end{itemize}
\end{defn}

\begin{lem}\label{intertwincompos}
	Composition of two intertwiners is an intertwiner.
\end{lem}
\begin{proof}
	Suppose $ V $, $ W $ and $ Z $ be unitary $\mathbb{G}$-$\mathbb{H}$ bimodules. Let $T : V \to W $ and $S : W \to Z $ be two intertwiners. We then have the following equations :
	\[\left(1 \otimes T\right) \circ \alpha_{\mathbb{G}, V } = \alpha_{\mathbb{G}, W} \circ T \ \ \ \ \ \ \left(T \otimes 1 \right) \circ \alpha_{V , \mathbb{H}} = \alpha_{  W, \mathbb{H}} \circ T \]
	\[\left(1 \otimes S\right) \circ \alpha_{\mathbb{G}, W } = \alpha_{\mathbb{G}, Z} \circ S \ \ \ \ \ \ \left(S \otimes 1 \right) \circ \alpha_{W , \mathbb{H}} = \alpha_{  Z, \mathbb{H}} \circ S \]
	Using the above equations we derive the following equations :
	\[ \left(1 \otimes S\circ T \right) \circ \alpha_{\mathbb{G}, V } = \left(1 \otimes S \right) \circ \alpha_{\mathbb{G}, W } \circ T = \alpha_{\mathbb{G}, Z} \circ S \circ T \]
	\[ \left(S \circ T \otimes 1 \right) \circ \alpha_{V , \mathbb{H}} = \left(S \otimes 1 \right) \circ \alpha_{  W, \mathbb{H}} \circ T = \alpha_{  Z, \mathbb{H}} \circ S \circ T \]
	Thus, $S \circ T : V \to Z $ is an intertwiner.
\end{proof}
We  now explore the monoidal structure of intertwiners.
\begin{prop}\label{intertwintensor}
	Suppose $V_1$ and $V_2$ be two unitary $\mathbb{G}$-$\mathbb{H}$ bimodules, and $W_1$ and $W_2$ be two unitary $\mathbb{H}$-$\mathbb{K}$ bimodules. Let $T : V_1 \to V_2$ be a $\mathbb{G}$-$\mathbb{H}$ intertwiner  and $S : W_1 \to W_2 $ be a $\mathbb{H}$-$\mathbb{K}$ intertwiner. Then, $T \otimes S : V_1 \otimes W_1 \to V_2 \otimes W_2 $ is a $\mathbb{G}$-$\mathbb{K}$ intertwiner. 
\end{prop}
\begin{proof}
	We have the following equations :
	\[\left(1 \otimes T\right) \circ \alpha_{\mathbb{G}, V_1 } = \alpha_{\mathbb{G}, V_2 } \circ T \ \ \ \ \ \ \left(T \otimes 1 \right) \circ \alpha_{V_1 , \mathbb{H}} = \alpha_{  V_2, \mathbb{H}} \circ T \]
	\[\left(1 \otimes S\right) \circ \alpha_{\mathbb{G}, W } = \alpha_{\mathbb{G}, Z} \circ S \ \ \ \ \ \ \left(S \otimes 1 \right) \circ \alpha_{W , \mathbb{H}} = \alpha_{  Z, \mathbb{H}} \circ S \]
	Now, \begin{align*}
		 \left(1 \otimes T \otimes S \right) \circ \alpha_{\mathbb{G}, V_1 \otimes W_1 } &= \left(1 \otimes T \otimes S \right) \circ \left(\alpha_{\mathbb{G},V_1} \otimes 1_{W_1}\right) \\
		 &= \left(1 \otimes 1_{V_2} \otimes S\right) \circ \left(\left(1 \otimes T\right) \circ \alpha_{\mathbb{G}, V_1} \otimes 1_{W_1}\right) \\
		 &= \left(1 \otimes 1_{V_2} \otimes S\right) \circ \left(\alpha_{\mathbb{G}, V_2 } \circ T \otimes 1_{W_1}\right) \\
		 &= \left(\alpha_{\mathbb{G}, V_2} \otimes 1_{W_2} \right) \circ \left(T \otimes S \right)\\
		 &= \alpha_{\mathbb{G}, V_2 \otimes W_2} \circ \left(T \otimes S \right)
		\end{align*}
	In a similar fashion we obtain, $\left(T \otimes S \otimes 1\right) \circ \alpha_{V_1 \otimes W_1, \mathbb{K}} = \alpha_{V_2 \otimes W_2, \mathbb{K}} \circ \left(T \otimes S\right)$. This concludes the proposition. 
\end{proof}

We are now in a position to define the $ 2 $-category of compact quantum groups
\begin{defn}
	The $ 2 $-category $\mcal G$ is defined as follows :
	\begin{itemize}
		\item $ 0 $-cells are compact quantum groups.
		\item $ 1 $-cells between compact quantum groups $\mathbb{G}$ and $\mathbb{H}$, consist of unitary $\mathbb{G}$-$\mathbb{H}$ bimodules.
		\item $ 2 $-cells between unitary $\mathbb{G}$-$\mathbb{H}$ bimodules $V$ and $W$ are $\mathbb{G}$-$\mathbb{H}$ intertwiners.  
	\end{itemize}
\end{defn}
The tensor product of two $ 1 $-cells $V \in \mcal G_1 (\mathbb{G},\mathbb{H})$ and $W \in \mcal G_1 (\mathbb{H},\mathbb{K})$ is given by the usual tensor product of Hilbert spaces $V \otimes W$. From \Cref{bimtensorproduct}, it indeed follows that $V \otimes W \in \mcal G_1 (\mathbb{G},\mathbb{K})$.

\begin{rem}\label{remstrict}
	In the subsequent discussions, we will treat the C*-tensor category of finite dimensional Hilbert spaces to be strict using the linear maps between the Hilbert spaces and the Cuntz algebra (see \cite{NT}).
\end{rem}

\begin{prop}
	$\mcal G$ forms a C*-2-category.
\end{prop}
 \begin{proof}
 	For two compact quantum groups $\mathbb{G}$ and $\mathbb{H}$, using \Cref{bimcorr} and \cite[Lemma 1.10]{KDC}, we get that each hom category $\mcal G_1 (\mathbb{G}, \mathbb{H})$ is a C*-category. Using \Cref{bimtensorproduct}, we see that $\mcal G$ is closed under tensor product of $ 1 $-cells, and from \Cref{intertwincompos} and \Cref{intertwintensor}, we get that $\mcal G$ is closed under composition and tensor product of $ 2 $-cells. 
 	To verify the compatibility between composition and tensor product of $ 2 $-cells in $\mcal G$, observe that, intertwiners being linear maps between finite dimensional Hilbert spaces trivially satisfy the compatibility condition. Finally, using \Cref{remstrict}, we take the associators and unitors of $\mcal G$ to be the identity maps.    
 \end{proof}

Our goal is to explore the Q-system completion of $\mcal G$. In the next result, we prove that $\mcal G$ is locally idempotent complete, and that will enable us to consider the Q-system completion of $\mcal G$.

\begin{thm}\label{idemthm}
	$\mcal G$ is locally idempotent complete.
\end{thm}
\begin{proof}
	Suppose $\mathbb{G}, \mathbb{H} \in \mcal G_0$ and $V \in \mcal G_1 (\mathbb{G}, \mathbb{H})$. Let $p$ be a projection in $\mcal G_2 (V,V)$. Using idempotent completeness in the category of Hilbert spaces, we get a Hilbert space $W$ and an isometry $i : W \to V$ such that $i i^* = p$. Now, $p$ being a $\mathbb{G}$-$\mathbb{H}$ intertwiner, gives us the following equations :
	\begin{equation}\label{equantionp}
		\left(1 \otimes p \right) \circ \alpha_{\mathbb{G}, V } = \alpha_{\mathbb{G}, V } \circ p \ \ \ \ \t{and} \ \ \ \ \left(p \otimes 1 \right) \circ \alpha_{V, \mathbb{H}} = \alpha_{V, \mathbb{H}} \circ p
	\end{equation}
We need to show that $ W $ is a unitary $\mathbb{G}$-$\mathbb{H}$ bimodule and $i : W \to V$ is a $\mathbb{G}$-$\mathbb{H}$ intertwiner. Define the maps $\alpha_{\mathbb{G}, W }$ and $\alpha_{W , \mathbb{H}}$ as follows :
\[\alpha_{\mathbb{G}, W } \coloneqq \left(1 \otimes i^* \right) \circ \alpha_{\mathbb{G}, V } \circ i : W \to C(\mathbb{G}) \otimes W \]
\[\alpha_{  W, \mathbb{H}} \coloneqq \left(i^* \otimes 1 \right) \circ \alpha_{V , \mathbb{H}} \circ i : W \to W \otimes C(\mathbb{H}) \]

We first show that $\mathbb{G}$ has a left co-action on $ W $ given by $\alpha_{\mathbb{G}, W }$. We have,
	 \begin{align*}
	\left(\Delta \otimes 1 \right) \circ \alpha_{\mathbb{G}, W } &= \left(\Delta \otimes 1 \right) \circ \left(1 \otimes i^* \right) \circ \alpha_{\mathbb{G}, V } \circ i \\
	&= \left(\Delta \otimes i^* \right) \circ \alpha_{\mathbb{G}, V } \circ i \\
	&= \left(1_{\mathbb{G}} \otimes 1_{\mathbb{G}} \otimes i^* \right) \circ \left(\Delta \otimes 1 \right) \circ \alpha_{\mathbb{G}, V } \circ i \\
	&= \left(1_{\mathbb{G}} \otimes 1_{\mathbb{G}} \otimes i^* \right) \circ \left(1 \otimes \alpha_{\mathbb{G}, V }\right) \circ \alpha_{\mathbb{G}, V } \circ i \ \  \t{(Since V is a left} \ \mathbb{G}-\t{module})	\end{align*}
\begin{align*}
	\left(1_{\mathbb{G}} \otimes \alpha_{\mathbb{G}, W }\right) \circ \alpha_{\mathbb{G}, W } &= \left(1_{\mathbb{G}} \otimes \left(1 \otimes i^* \right) \circ \alpha_{\mathbb{G}, V } \circ i \right) \circ \left(1 \otimes i^* \right) \circ \alpha_{\mathbb{G}, V } \circ i \\
	&= \left(1_{\mathbb{G}} \otimes 1_{\mathbb{G}} \otimes i^* \right) \circ \left(1 \otimes \alpha_{\mathbb{G}, V }\right) \circ (1 \otimes p) \circ \alpha_{\mathbb{G}, V } \circ i \ \ \ \ \t{(Since} \ \ i i^* =p) \\
	&= \left(1_{\mathbb{G}} \otimes 1_{\mathbb{G}} \otimes i^* \right) \circ \left(1 \otimes \alpha_{\mathbb{G}, V }\right) \circ \alpha_{\mathbb{G}, V } \circ p \circ i  \ \ \ \ \t{(By \Cref{equantionp})} \\
	&= \left(1_{\mathbb{G}} \otimes 1_{\mathbb{G}} \otimes i^* \right) \circ \left(1 \otimes \alpha_{\mathbb{G}, V }\right) \circ \alpha_{\mathbb{G}, V } \circ p \circ i \\
	&= \left(1_{\mathbb{G}} \otimes 1_{\mathbb{G}} \otimes i^* \right) \circ \left(1 \otimes \alpha_{\mathbb{G}, V }\right) \circ \alpha_{\mathbb{G}, V } \circ i
\end{align*}
Therefore,  $\left(\Delta \otimes 1 \right) \circ \alpha_{\mathbb{G}, W } = \left(1_{\mathbb{G}} \otimes \alpha_{\mathbb{G}, W }\right) \circ \alpha_{\mathbb{G}, W } $. In a similar way, it can be shown that $\mathbb{H}$ has a right co-action on $W$ given by $\alpha_{  W, \mathbb{H}}$.

We now establish the compatibility condition of left and right co-actions on $W$.
\begin{align*}
	\left(1 \otimes \alpha_{  W, \mathbb{H}}\right) \circ \alpha_{\mathbb{G}, W } &= \left(1 \otimes \left(i^* \otimes 1 \right) \circ \alpha_{V , \mathbb{H}} \circ i \right) \circ (\left(1 \otimes i^* \right) \circ \alpha_{\mathbb{G}, V } \circ i) \\
	&= (1\otimes i^*) \circ (1\otimes \alpha_{V , \mathbb{H}}) \circ (1 \otimes i i^*) \circ \alpha_{\mathbb{G}, V } \circ i \\
	&= (1\otimes i^*) \circ (1\otimes \alpha_{V , \mathbb{H}}) \circ (1 \otimes p) \circ \alpha_{\mathbb{G}, V } \circ i \ \ \t{(Since} \ \ i i^* = p) \\
	&= (1\otimes i^*) \circ (1\otimes \alpha_{V , \mathbb{H}}) \circ \alpha_{\mathbb{G}, V } \circ p \circ i \ \ \ \ (\t{By \Cref{equantionp}}) \\
	&= (1\otimes i^*) \circ (1\otimes \alpha_{V , \mathbb{H}}) \circ \alpha_{\mathbb{G}, V } \circ i \\ 
	&= (1 \otimes i^*) \circ (\alpha_{\mathbb{G}, V } \otimes 1) \circ \alpha_{V , \mathbb{H}} \circ i \ \ \ \ (\t{By condition (iii) of \Cref{bimoduledefn}})
\end{align*}

\begin{align*}
	\left(\alpha_{\mathbb{G}, W } \otimes 1\right) \circ \alpha_{  W, \mathbb{H}} &= \left(\left(1 \otimes i^* \right) \circ \alpha_{\mathbb{G}, V } \circ i \otimes 1 \right) \circ (\left(i^* \otimes 1 \right) \circ \alpha_{V , \mathbb{H}} \circ i) \\
	&= (1 \otimes i^*) \circ (\alpha_{\mathbb{G}, V } \otimes 1) \circ (ii^* \otimes 1) \circ \alpha_{V , \mathbb{H}} \circ i \\
	&= (1 \otimes i^*) \circ (\alpha_{\mathbb{G}, V } \otimes 1) \circ (p \otimes 1) \circ \alpha_{V , \mathbb{H}} \circ i \\
	&= (1 \otimes i^*) \circ (\alpha_{\mathbb{G}, V } \otimes 1) \circ \alpha_{V , \mathbb{H}} \circ p \circ i \ \ \ \ \t{(By \Cref{equantionp})} \\
	&= (1 \otimes i^*) \circ (\alpha_{\mathbb{G}, V } \otimes 1) \circ \alpha_{V , \mathbb{H}} \circ i
\end{align*}
Therefore, $\left(1 \otimes \alpha_{  W, \mathbb{H}}\right) \circ \alpha_{\mathbb{G}, W } = \left(\alpha_{\mathbb{G}, W } \otimes 1\right) \circ \alpha_{  W, \mathbb{H}} $

Finally, we are left with verifying condition (iv) of \Cref{bimoduledefn} for $\alpha_{\mathbb{G}, W }$ and $\alpha_{  W, \mathbb{H}}$. For $w_1 ,  w_2 \in W$, we have, 
\begin{align*}
	\lab \alpha_{\mathbb{G}, W }(w_1) , \alpha_{\mathbb{G}, W }(w_1) \rab_{C(\mathbb{G})} &= \lab \left(1 \otimes i^* \right) \circ \alpha_{\mathbb{G}, V } \circ i (w_1) , \left(1 \otimes i^* \right) \circ \alpha_{\mathbb{G}, V } \circ i (w_2) \rab_{C(\mathbb{G})} \\
	&= \lab \alpha_{\mathbb{G}, V } \circ i (w_1) , \left(1 \otimes p \right) \circ \alpha_{\mathbb{G}, V } \circ i (w_2) \rab_{C(\mathbb{G})}  \comments{\ \\ \ \t{(Since} \ \ i i^* = p) }\\
	&= \lab \alpha_{\mathbb{G}, V } \circ i (w_1) , \alpha_{\mathbb{G}, V } \circ p \circ i (w_2) \rab_{C(\mathbb{G})} \\
	&= \lab \alpha_{\mathbb{G}, V } \circ i (w_1) , \alpha_{\mathbb{G}, V } \circ i (w_2) \rab_{C(\mathbb{G})} \\
	&= \lab i (w_1) , i (w_2) \rab \ \ \ \ \t{(By condition (iv) of \Cref{bimoduledefn})} \\
	&= \lab w_1, w_2 \rab .
\end{align*}
Similarly, we get $\lab \alpha_{\mathbb{G}, W }(w_1) , \alpha_{\mathbb{G}, W }(w_1) \rab_{C(\mathbb{G})} = \lab w_1 , w_2 \rab $. This concludes the theorem .   
\end{proof}

\begin{rem}\label{qsysrem}
	\Cref{idemthm}, allows one to consider Q-system completion of $\mcal G$. By \cite{CPJP}, a locally idempotent complete C*-2-category $\mcal C$ is Q-system complete if and only if every Q-system in $\mcal C$ splits in $\mcal C$. We would like to remark that $\mcal G$ is far from being Q-system complete. Consider any compact quantum group $\mathbb{G}$ and consider $\C^n$, where $n$ is not a perfect square, equipped with trivial left $\mathbb{G}$ co-action and a trivial right $\mathbb{G}$ co-action (see \Cref{example}). Now, it is trivial to verify that $\C^n$ becomes a Q-system in $\mcal G$ but it cannot split in $\mcal G$. Therefore, in the next section, we try to describe $\textbf{QSys}(\mcal G)$, and that led us to consider the category of \textit{quantum bi-elements}.
\end{rem}

\section{Quantum bi-elements}\label{qbl}

The foundational idea of noncommutative topology is to generalise the correspondence between topological spaces and commutative algebras by considering noncommutative algebras in light of Gelfand duality. In \cite{MRV}, the category of finite sets and functions has been quantised to form quantum sets and quantum functions. This `quantisation' approach categorifies finite sets and functions into a $ 2 $-category. This further reveals their higher compositional structure. Quantum elements \cite{MRV} were formulated as objects of the corresponding quantum set. Our formulation of \textit{quantum bi-elements} can be regarded as a two-sided version of quantum elements. Towards the end of this section, we construct non-unital quantum functions from quantum bi-elements and provide a description of Q-system completion of $\mcal G$ as full subcategory of the category of quantum bi-elements.   

\begin{defn}\label{qbelmntdefn}
	Suppose $A$ and $B$ are Q-systems in $\textbf{Hilb}$. A \textit{quantum bi-element} of the pair $(A , B)$ is given by a triplet $(H,Q_1,Q_2)$, where $H$ is a finite dimensional Hilbert space, and $Q_1 : H \to A \otimes H$ and $Q_2 : H \to H \otimes B$ are linear maps, represented as \[Q_1 \ = \ \raisebox{-2mm}{\begin{tikzpicture}[rotate=180]
			\draw (0,0) to (0,.6);%X
			\draw[red][in=60,out=90,looseness=1] (.5,0) to (0,.3);%Q
			\node[scale=.8] at (0,.3) {$\color{red}{\bullet}$};
	\end{tikzpicture}} \ \ , \ \
	Q_2 \ = \ \raisebox{-2mm}{\begin{tikzpicture}[rotate=180]
			\draw (0,0) to (0,.6);%X
			\draw[blue][in=120,out=90,looseness=1] (-.5,0) to (0,.3);%P
			\node[scale=.8] at (0,.3) {$\color{blue}{\bullet}$};
	\end{tikzpicture}} \ \ , \ \ Q_1^* \ =  \ \raisebox{-2mm}{\begin{tikzpicture}
\draw (0,0) to (0,.6);%X
\draw[red][in=120,out=90,looseness=1] (-.5,0) to (0,.3);%P
\node[scale=.8] at (0,.3) {$\color{red}{\bullet}$};
\end{tikzpicture}} \ \ , \ \ Q_2^* = \raisebox{-2mm}{\begin{tikzpicture}
\draw (0,0) to (0,.6);%X
\draw[blue][in=60,out=90,looseness=1] (.5,0) to (0,.3);%Q
\node[scale=.8] at (0,.3) {$\color{blue}{\bullet}$};
\end{tikzpicture}}\]
where $H,A$ and $ B $ are represented by black, red and blue strands respectively. The triplet $(H,Q_1,Q_2)$ satisfy the following :
\begin{itemize}
	\item [(1)] \reflectbox{\rotatebox[origin=c]{180}{\raisebox{-4mm}{\begin{tikzpicture}
			\draw (0,0) to (0,1.2);
			\draw[red][in=120,out=90,looseness=1] (-.5,0) to (0,.5);
			\draw[red][in=120,out=90,looseness=1] (-.7,0) to (0,.7);
			\node[scale=.8] at (0,.52) {$\color{red}{\bullet}$};
			\node[scale=.8] at (0,.72) {$\color{red}{\bullet}$};
			\node at (.6,.6) {$=$};
			\draw[red,in=90,out=90,looseness=2] (1,0) to (1.5,0);
			\node[scale=.7] at (1.25,.27) {$\red{\bullet}$};
			\draw[red,in=120,out=90,looseness=1] (1.25,.27) to (1.75,.57);
			\draw (1.75,0) to (1.75,1.2);
			\node[scale=.8] at (1.75,.59) {$\color{red}{\bullet}$};
	\end{tikzpicture}}}} \hspace*{2mm}, \hspace*{2mm}
	\reflectbox{\rotatebox[origin=c]{180}{\raisebox{-4mm}{\begin{tikzpicture}
			\draw (0,0) to (0,1.2);
			\draw[blue][in=60,out=90,looseness=1] (.5,0) to (0,.5);
			\draw[blue][in=60,out=90,looseness=1] (.7,0) to (0,.7);
			\node[scale=.8] at (0,.52) {$\color{blue}{\bullet}$};
			\node[scale=.8] at (0,.72) {$\color{blue}{\bullet}$};
			\node at (1,.6) {$=$};
			\draw (1.3,0) to (1.3,1.2);
			\draw[blue,in=90,out=90,looseness=2] (1.55,0) to (2.05,0);
			\draw[blue,in=60,out=90,looseness=1] (1.8,.27) to (1.3,.57);
			\node[scale=.7] at (1.8,.27) {$\color{blue}{\bullet}$};
			\node[scale=.8] at (1.3,.57) {$\color{blue}{\bullet}$};
	\end{tikzpicture}}}} \hspace*{2mm} and \hspace*{2mm}
	\reflectbox{\rotatebox[origin=c]{180}{\raisebox{-4mm}{\begin{tikzpicture}
			\draw (0,0) to (0,1);%X
			\draw[red][in=120,out=90,looseness=1] (-.5,0) to (0,.6);%P
			\draw[blue][in=60,out=90,looseness=1] (.5,0) to (0,.3);%Q
			\node[scale=.8] at (0,.6) {$\color{red}{\bullet}$};
			\node[scale=.8] at (0,.3) {$\color{blue}{\bullet}$};
			\node at (.7,.6) {$=$};
			\draw (1.5,0) to (1.5,1);%X
			\draw[red][in=120,out=90,looseness=1] (1,0) to (1.5,.3);%P
			\draw[blue][in=60,out=90,looseness=1] (2,0) to (1.5,.6);%Q
			\node[scale=.8] at (1.5,.6) {$\color{blue}{\bullet}$};
			\node[scale=.8] at (1.5,.3) {$\color{red}{\bullet}$};
	\end{tikzpicture}}}}
\item [(2)] \reflectbox{\rotatebox[origin=c]{180} {\raisebox{-4mm}{\begin{tikzpicture}
		\draw (0,0) to (0,1);%X
		\draw[red][in=120,out=90,looseness=1] (-.5,.2) to (0,.6);%P
		\node[scale=.7] at (-.5,.2) {$\red{\bullet}$};
		\node[scale=.8] at (0,.6) {$\color{red}{\bullet}$};
		\node at (.4,.6) {$=$};
		\draw (.8,0) to (.8,1);
		\node at (1.2,.6) {$=$};
		\draw (1.6,0) to (1.6,1);
		\draw[blue][in=60,out=90,looseness=1] (2.1,.2) to (1.6,.6);%Q
		\node[scale=.7] at (2.1,.2) {$\color{blue}{\bullet}$};
		\node[scale=.8] at (1.6,.6) {$\color{blue}{\bullet}$};
\end{tikzpicture}}}}

\item [(3)] $\raisebox{-4mm}{\begin{tikzpicture}
		\draw (0,0) to (0,.8);%X
		\draw[red][in=120,out=90,looseness=1] (-.5,0) to (0,.3);%P
		\node[scale=.8] at (0,.3) {$\color{red}{\bullet}$};
\end{tikzpicture}} = \raisebox{-6mm}{\begin{tikzpicture}
\draw (0,0) to (0,1.2);%X
\draw[red][in=-120,out=-90,looseness=1] (-.3,.5) to (0,.3);
\draw[red,in=90,out=90,looseness=2] (-.7,.5) to (-.3,.5);
\draw[red] (-.7,.5) to (-.7,0);
\draw[red] (-.5,.7) to (-.5,1);
\node[scale=.7] at (-.5,.72) {$\red{\bullet}$};
\node[scale=.8] at (0,.3) {$\color{red}{\bullet}$};
\node[scale=.8] at (-.5,1.05) {$\red{\bullet}$};
\end{tikzpicture}}$ \hspace*{2mm} and $\hspace*{2mm}
\raisebox{-2mm}{\begin{tikzpicture}
		\draw (0,0) to (0,.8);%X
		\draw[blue][in=60,out=90,looseness=1] (.5,0) to (0,.3);%Q
		\node[scale=.8] at (0,.3) {$\color{blue}{\bullet}$};
\end{tikzpicture}} = \raisebox{-6mm}{\begin{tikzpicture}
\draw (0,0) to (0,1.2);
\draw[blue][in=-60,out=-90,looseness=1] (.3,.5) to (0,.3);
\draw[blue][in=90,out=90,looseness=2] (.3,.5) to (.7,.5);
\draw[blue] (.5,.7) to (.5,1); 
\draw[blue] (.7,.5) to (.7,0);
\node[scale=.8] at (0,.3) {$\color{blue}{\bullet}$};
\node[scale=.7] at (.5,.7) {$\color{blue}{\bullet}$};
\node[scale=.8] at (.5,1.05) {$\color{blue}{\bullet}$};
\end{tikzpicture}}$ 
\end{itemize}
\end{defn}

\begin{ex}
	Suppose $(A,m,i)$ be a Q-system in \textbf{Hilb}. Clearly, it follows that, $(A,m^*,m^*)$ is a quantum bi-element of the pair $(A,A)$.
\end{ex}

\begin{ex}\label{qbleg}
	A quantum bi-element of $(A,\C)$ is a quantum element (in the sense of \cite{MRV}) of $A$.
\end{ex}

It turns out that the conditions for being a quantum bi-element $(H,Q_1,Q_2)$ of the pair $(A,B)$, enforces the maps $Q_1$ and $Q_2$ to be isometries. Before we prove this result, we first prove the following lemma which might be known to experts, nevertheless we give a proof.    

\begin{lem}\label{frobeniuslem}
	Suppose $A$ and $B$ are Q-systems in $\normalfont\textbf{Hilb}$ and let $(H,Q_1,Q_2)$ be a quantum bi-element of $(A,B)$. Then we have the following :
	\begin{itemize}
		\item [(i)] $\raisebox{-6mm}{\begin{tikzpicture}
				\draw (0,0) to (0,1.2);%X
				\draw[red][in=-120,out=-90,looseness=1] (-.3,.5) to (0,.3);
				\draw[red,in=90,out=90,looseness=2] (-.7,.5) to (-.3,.5);
				\draw[red] (-.7,.5) to (-.7,0);
				\draw[red] (-.5,.7) to (-.5,1);
				\node[scale=.7] at (-.5,.72) {$\red{\bullet}$};
				\node[scale=.8] at (0,.3) {$\color{red}{\bullet}$};
				\node at (.4,.6) {$=$};
				\draw (1.2,0) to (1.2,1.2);
				\draw[red][in=-120,out=-90,looseness=1] (.8,1.2) to (1.2,.8);
				\draw[red][in=120,out=90,looseness=1] (.8,0) to (1.2,.4);
				\node[scale=.8] at (1.2,.8) {$\color{red}{\bullet}$};
				\node[scale=.8] at (1.2,.4) {$\color{red}{\bullet}$};
				\node at (1.6,.6) {$=$};
				\draw (2.8,0) to (2.8,1.2);
				\draw[red][in=120,out=90,looseness=1] (2.4,.4) to (2.8,.8);
				\draw[red,in=-90,out=-90,looseness=2] (2.4,.4) to (2,.4);
				\draw[red] (2,.4) to (2,1);
				\draw[red] (2.2,.2) to (2.2,-.1);
				\node[scale=.7] at (2.2,.15) {$\red{\bullet}$};
				\node[scale=.8] at (2.8,.8) {$\color{red}{\bullet}$};
		\end{tikzpicture}}$
	\item [(ii)] $\raisebox{-6mm}{\begin{tikzpicture}
			\draw (0,0) to (0,1.2);
			\draw[blue][in=-60,out=-90,looseness=1] (.3,.5) to (0,.3);
			\draw[blue][in=90,out=90,looseness=2] (.3,.5) to (.7,.5);
			\draw[blue] (.5,.7) to (.5,1); 
			\draw[blue] (.7,.5) to (.7,0);
			\node[scale=.8] at (0,.3) {$\color{blue}{\bullet}$};
			\node[scale=.7] at (.5,.7) {$\color{blue}{\bullet}$};
			\node at (1.1,.6) {$=$};
			\draw (1.5,0) to (1.5,1.2);
			\draw[blue][in=-60,out=-90,looseness=1] (1.9,1.2) to (1.5,.8);
			\draw[blue][in=60,out=90,looseness=1] (1.9,0) to (1.5,.4);
			\node[scale=.8] at (1.5,.8) {$\color{blue}{\bullet}$};
			\node[scale=.8] at (1.5,.4) {$\color{blue}{\bullet}$};
			\node at (2.3,.6) {$=$};
			\draw (2.7,0) to (2.7,1.2);
			\draw[blue][in=60,out=90,looseness=1] (3,.4) to (2.7,.7);
			\draw[blue,in=-90,out=-90,looseness=2] (3,.4) to (3.4,.4);
			\draw[blue] (3.4,.4) to (3.4,1);
			\draw[blue] (3.2,.2) to (3.2,0);
			\node[scale=.7] at (3.2,.15) {$\color{blue}{\bullet}$};
			\node[scale=.8] at (2.7,.7) {$\color{blue}{\bullet}$}; 
	\end{tikzpicture}}$
	\end{itemize}
\end{lem}

\begin{proof}\
	
	\begin{itemize}
		\item [(i)] Using condition (3) of \Cref{qbelmntdefn}, we have $$ \raisebox{-4mm}{\begin{tikzpicture}
			\draw (1.2,0) to (1.2,1.2);
			\draw[red][in=-120,out=-90,looseness=1] (.8,1.2) to (1.2,.8);
			\draw[red][in=120,out=90,looseness=1] (.8,0) to (1.2,.4);
			\node[scale=.8] at (1.2,.8) {$\color{red}{\bullet}$};
			\node[scale=.8] at (1.2,.4) {$\color{red}{\bullet}$};
		\end{tikzpicture}} = \raisebox{-4mm}{\begin{tikzpicture}
	\draw (0,0) to (0,1.2);
	\draw[red,in=-120,out=-90,looseness=1] (-.4,1.2) to (0,.8);
	\draw[red,in=-120,out=-90,looseness=1] (-.5,.7) to (0,.3);
	\draw[red,in=90,out=90,looseness=2] (-.5,.7) to (-.9,.7);
	\draw[red] (-.7,.9) to (-.7,1.3);
	\draw[red] (-.9,.7) to (-.9,0);
	\node[scale=.8] at (-.7,.9) {$\red{\bullet}$};
	\node[scale=.8] at (-.7,1.3) {$\red{\bullet}$};
	\node[scale=.8] at (0,.8) {$\red{\bullet}$};
	\node[scale=.8] at (0,.3) {$\red{\bullet}$};
\end{tikzpicture}} = \raisebox{-4mm}{\begin{tikzpicture}
\draw (0,0) to (0,1.5);
\draw[red,in=-120,out=-90,looseness=1] (-.5,.8) to (0,.5);
\draw[red,in=-90,out=-90,looseness=2] (-.7,1) to (-.3,1);
\draw[red,in=90,out=90,looseness=2] (-.7,1) to (-1.1,1);
\draw[red] (-.9,1.2) to (-.9,1.5);
\draw[red] (-1.1,1) to (-1.1,0);
\draw[red] (-.3,1) to (-.3,1.5);
\node[scale=.7] at (-.5,.75) {$\red{\bullet}$};
\node[scale=.7] at (-.9,1.2) {$\red{\bullet}$};
\node[scale=.7] at (-.9,1.5) {$\red{\bullet}$};
\node[scale=.7] at (0,.5) {$\red{\bullet}$};
\end{tikzpicture}} = \raisebox{-4mm}{\begin{tikzpicture}
\draw (0,0) to (0,1.2);%X
\draw[red][in=-120,out=-90,looseness=1] (-.3,.5) to (0,.3);
\draw[red,in=90,out=90,looseness=2] (-.7,.5) to (-.3,.5);
\draw[red] (-.7,.5) to (-.7,0);
\draw[red] (-.5,.7) to (-.5,1.2);
\node[scale=.7] at (-.5,.72) {$\red{\bullet}$};
\node[scale=.8] at (0,.3) {$\color{red}{\bullet}$};
\end{tikzpicture}} $$
The second equality follows from the first equation of condition (1) of \Cref{qbelmntdefn}. Unitality and Frobenius condition of the Q-system $A$ reveals the last equality. The other equation of (i) can be proved similarly.
 
 \item [(ii)] The proof is similar to that of (i).
	\end{itemize}
\end{proof}

\begin{prop}\label{isoprop}
	Suppose $A$ and $B$ are Q-systems in $\normalfont\textbf{Hilb}$ and let $(H,Q_1,Q_2)$ be a quantum bi-element of $(A,B)$. Then $Q_1$ and $Q_2$ are isometries.
\end{prop}
\begin{proof}
	From \Cref{frobeniuslem} and condition (3) of \Cref{qbelmntdefn}, we obtain, $$ Q_1 ^* Q_1 = \raisebox{-12mm}{\begin{tikzpicture}
			\draw (1.2,-.5) to (1.2,1.7);
			\draw[red][in=-120,out=-90,looseness=1] (.8,1.2) to (1.2,.8);
			\draw[red][in=120,out=90,looseness=1] (.8,0) to (1.2,.4);
			\draw[red,in=90,out=90,looseness=2] (.8,1.2) to (.4,1.2);
			\draw[red,in=-90,out=-90,looseness=2] (.8,0) to (.4,0);
			\draw[red] (.4,0) to (.4,1.2);
			\draw[red] (.6,-.25) to (.6,-.5);
			\draw[red] (.6,1.4) to (.6,1.7);
			\node[scale=.8] at (.6,-.25) {$\red{\bullet}$};
			\node[scale=.8] at (.6,1.4) {$\red{\bullet}$};
			\node[scale=.8] at (.6,-.5) {$\red{\bullet}$};
			\node[scale=.8] at (.6,1.7) {$\red{\bullet}$};
 			\node[scale=.8] at (1.2,.8) {$\color{red}{\bullet}$};
			\node[scale=.8] at (1.2,.4) {$\color{red}{\bullet}$};
	\end{tikzpicture}} = \raisebox{-8mm}{\begin{tikzpicture}
\draw (0,-.5) to (0,1.4);%X
\draw[red][in=-120,out=-90,looseness=1] (-.3,.5) to (0,.3);
\draw[red,in=90,out=90,looseness=2] (-.7,.5) to (-.3,.5);
\draw[red,in=90,out=90,looseness=2] (-.5,.7) to (-1,.7);
\draw[red,in=-90,out=-90,looseness=2] (-1,0) to (-.7,0);
\draw[red] (-.7,.5) to (-.7,0);
\draw[red] (-1,0) to (-1,.7);
\draw[red] (-.75,1) to (-.75,1.3);
\draw[red] (-.85,-.5) to (-.85,-.2);
\node[scale=.7] at (-.85,-.5) {$\red{\bullet}$};
\node[scale=.7] at (-.85,-.2) {$\red{\bullet}$};
\node[scale=.7] at (-.75,1.3) {$\red{\bullet}$};
\node[scale=.7] at (-.5,.72) {$\red{\bullet}$};
\node[scale=.8] at (0,.3) {$\color{red}{\bullet}$};
\node[scale=.7] at (-.75,1) {$\color{red}{\bullet}$}; 
\end{tikzpicture}} = \raisebox{-2mm}{\reflectbox{\rotatebox[origin=c]{180}{\begin{tikzpicture}
\draw (0,0) to (0,1);%X
\draw[red][in=120,out=90,looseness=1] (-.5,.2) to (0,.6);%P
\node[scale=.7] at (-.5,.2) {$\red{\bullet}$};
\node[scale=.8] at (0,.6) {$\color{red}{\bullet}$};
%\node at (.4,.6) {$=$};
%\draw (.8,0) to (.8,1);
\end{tikzpicture}}}} = \raisebox{-4mm}{\begin{tikzpicture}
\draw (0,0) to (0,1.4);
\end{tikzpicture}}  $$
The third equality follows from the unitality, separability, Frobenius condition of the Q-system $A$. The last equality is just the first equation of condition (2) of \Cref{qbelmntdefn}. In a similar manner, it can be proved that $Q_2$ is also an isometry.
\end{proof}

Quantum functions underlie various notions of `quantum morphism' which is prevalent in quantum information theory and noncommutative topology. They can be understood as perfect quantum strategies for a certain `function game'. We provide a construction of non-unital quantum functions using quantum bi-elements.

\begin{defn}\cite{MRV} \label{qfndefn}
	A \textit{quantum function} between Q-systems $A$ and $B$ is a pair $(H,P)$, where $H$ is a finite dimensional Hilbert space and $P : H \otimes A \to B \otimes H$ satisfying the following :
	\begin{equation}
		\raisebox{-6mm}{\begin{tikzpicture}
			\draw (-.3,0) to (-.3,-.8);
			\draw[red] (.3,0) to (.3,-.8);
			\draw (.3,0) to (.3,.9) ;
			\draw[blue] (-.3,0) to (-.3,.5);
			\draw[blue,in=-90,out=-90,looseness=2] (-.6,.9) to (0,.9);
			\node[scale=.8] at (-.3,.55) {$\color{blue}{\bullet}$};
			\node[draw,thick, rounded corners, fill=white,minimum width=30] at (0,0) {$P$};
		\end{tikzpicture}} \ = \ \raisebox{-8mm}{\begin{tikzpicture}
		\draw (-.3,0) to (-.3,-.8);
		\draw[red] (.9,-.2) to (.9,.8);
		\draw (.9,.8) to (.9,1.5);
		\draw (.3,0) to (.3,.6) ;
		\draw[blue] (.3,.8) to (.3,1.5);
		\draw[blue] (-.3,0) to (-.3,1.5);
		\draw[red,in=-90,out=-90,looseness=2] (.3,-.2) to (.9,-.2);
		\draw[red] (.6,-.55) to (.6,-.8);
		%\node[scale=.8] at (-.3,.55) {$\color{blue}{\bullet}$};
		\node[draw,thick, rounded corners, fill=white,minimum width=30] at (0,0) {$P$};
		\node[draw,thick, rounded corners, fill=white,minimum width=30] at (.6,.8) {$P$};
		\node[scale=.8] at (.6,-.55) {$\red{\bullet}$};
	\end{tikzpicture}} \ \  \ \ \ \ \ \ \ \ 
\raisebox{-6mm}{\begin{tikzpicture}
		\draw (-.3,0) to (-.3,-.8);
		\draw[red] (.3,0) to (.3,-.8);
		\draw (.3,0) to (.3,.9) ;
		\draw[blue] (-.3,0) to (-.3,.75);
		%\draw[blue,in=-90,out=-90,looseness=2] (-.6,.9) to (0,.9);
		\node[scale=.8] at (-.3,.75) {$\color{blue}{\bullet}$};
		\node[draw,thick, rounded corners, fill=white,minimum width=30] at (0,0) {$P$};
\end{tikzpicture}} \ = \ \raisebox{-4mm}{\begin{tikzpicture}
\draw (0,0) to (0,1);
\draw[red] (.5,0) to (.5,.8);
\node[scale=.8] at (.5,.8) {$\red{\bullet}$};
\end{tikzpicture}} \ \ \ \ \ \ \ \ \ \
 \raisebox{-6mm}{\begin{tikzpicture}
 		\draw[blue] (-.3,0) to (-.3,-.8);
 		\draw (.3,0) to (.3,-.8);
 		\draw[red] (.3,0) to (.3,.8) ;
 		\draw (-.3,0) to (-.3,.8);
 		\node[draw,thick, rounded corners, fill=white,minimum width=30] at (0,0) {$P^*$};
 \end{tikzpicture}} \ = \ \raisebox{-6mm}{\begin{tikzpicture}
 \draw (-.3,0) to (-.3,-.8);
 \draw[red,in=-90,out=-90,looseness=2] (.3,-.2) to (.8,-.2);
 \draw (.3,0) to (.3,.9) ;
 \draw[blue,in=90,out=90,looseness=2] (-.3,.2) to (-.8,.2);
 \draw[blue] (-.55,.45) to (-.55,.8);
 \draw[blue] (-.8,.2) to (-.8,-.8);
 \draw[red] (.55,-.8) to (.55,-.5);
 \draw[red] (.8,-.2) to (.8,.9);
 %\draw[blue,in=-90,out=-90,looseness=2] (-.6,.9) to (0,.9);
 \node[scale=.8] at (-.55,.45) {$\color{blue}{\bullet}$};
 \node[scale=.8] at (-.55,.8) {$\color{blue}{\bullet}$};
 \node[scale=.8] at (.55,-.5) {$\color{red}{\bullet}$};
 \node[scale=.8] at (.55,-.8) {$\color{red}{\bullet}$};
 \node[draw,thick, rounded corners, fill=white,minimum width=30] at (0,0) {$P$};
\end{tikzpicture}}
	\end{equation}
\end{defn}
We will call a quantum function to be `non-unital' if it does not satisfy the second condition of \Cref{qfndefn}. \comments{We say that $P$ is a quantum function from $A$ to $B$}

Quantum bi-elements give rise to non-unital quantum functions.
Suppose $A$ and $B$ are Q-systems in $\normalfont\textbf{Hilb}$ and let $(H,Q_1,Q_2)$ be a quantum bi-element of $(A,B)$. Using the quantum bi-element $(H,Q_1,Q_2)$ we can construct a non-unital quantum function from $B$ to $A$  as follows. Define 
\[ P \coloneqq \raisebox{-4mm}{\begin{tikzpicture}
	\draw (0,0) to (0,1.2);
	\draw[blue,in=120,out=90,looseness=1] (.5,0) to (0,.4);
	\draw[red,in=-120,out=-90,looseness=1] (-.5,1.2) to (0,.8);
	\node[scale=.8] at (0,.4) {$\color{blue}{\bullet}$};
	\node[scale=.8] at (0,.8) {$\red{\bullet}$};
\end{tikzpicture}} : H \otimes B \to A \otimes H \]
Our claim is that $P$ will be a non-unital quantum function from $B$ to $A$. Before we prove this, we first prove the following lemma :
\begin{lem}\label{associativelem}
	With the above notations, we have, \[ P = \raisebox{-4mm}{\begin{tikzpicture}
			\draw (0,0) to (0,1.2);
			\draw[blue,in=120,out=90,looseness=1] (.5,0) to (0,.4);
			\draw[red,in=-120,out=-90,looseness=1] (-.5,1.2) to (0,.8);
			\node[scale=.8] at (0,.4) {$\color{blue}{\bullet}$};
			\node[scale=.8] at (0,.8) {$\red{\bullet}$};
	\end{tikzpicture}} = \raisebox{-4mm}{\begin{tikzpicture}
	\draw (0,0) to (0,1.2);
	\draw[blue,in=120,out=90,looseness=1] (.5,.4) to (0,.8);
	\draw[red,in=-120,out=-90,looseness=1] (-.5,.8) to (0,.4);
	\draw[red] (-.5,.8) to (-.5,1.2);
	\draw[blue] (.5,.4) to (.5,0);
	\node[scale=.8] at (0,.8) {$\color{blue}{\bullet}$};
	\node[scale=.8] at (0,.4) {$\red{\bullet}$};
\end{tikzpicture}}  \]
\end{lem}
\begin{proof}
	It easily follows from the third equation of condition (1) of \Cref{qbelmntdefn} and condition (3) of \Cref{qbelmntdefn}.
\end{proof}

\begin{prop}
	$P$ is a non-unital quantum function from $B$ to $A$.
\end{prop}
\begin{proof}
	We have,
	$$\raisebox{-12mm}{\begin{tikzpicture}
			\draw (-.3,0) to (-.3,-.8);
			\draw[blue] (.9,-.2) to (.9,.8);
			\draw (.9,.8) to (.9,1.5);
			\draw (.3,0) to (.3,.6) ;
			\draw[red] (.3,.8) to (.3,1.5);
			\draw[red] (-.3,0) to (-.3,1.5);
			\draw[blue,in=-90,out=-90,looseness=2] (.3,-.2) to (.9,-.2);
			\draw[blue] (.6,-.55) to (.6,-.8);
			%\node[scale=.8] at (-.3,.55) {$\color{blue}{\bullet}$};
			\node[draw,thick, rounded corners, fill=white,minimum width=30] at (0,0) {$P$};
			\node[draw,thick, rounded corners, fill=white,minimum width=30] at (.6,.8) {$P$};
			\node[scale=.8] at (.6,-.55) {$\color{blue}{\bullet}$};
	\end{tikzpicture}} = \raisebox{-12mm}{\begin{tikzpicture}
	\draw (0,-.7) to (0,2);
	\draw[blue,in=120,out=90,looseness=1] (.5,0) to (0,.4);
	\draw[blue,in=120,out=90,looseness=1] (.9,.7) to (0,1.2);
	\draw[red,in=-120,out=-90,looseness=1] (-.9,1.3) to (0,.8);
	\draw[red,in=-120,out=-90,looseness=1] (-.5,2) to (0,1.6);
	\draw[blue] (.9,.7) to (.9,0);
	\draw[blue,in=-90,out=-90,looseness=2] (.5,0) to (.9,0);
	\draw[blue] (.7,-.3) to (.7,-.7);
	\node[scale=.8] at (.7,-.25) {$\color{blue}{\bullet}$};
	\node[scale=.8] at (0,.4) {$\color{blue}{\bullet}$};
	\node[scale=.8] at (0,.8) {$\red{\bullet}$};
	\node[scale=.8] at (0,1.25) {$\color{blue}{\bullet}$};
	\node[scale=.8] at (0,1.6) {$\red{\bullet}$};
\end{tikzpicture}} = \raisebox{-12mm}{\begin{tikzpicture}
\draw (0,-.7) to (0,2);
\draw[blue] (.9,.2) to (.9,0);
\draw[blue] (.7,-.3) to (.7,-.7);
\draw[blue,in=120,out=90,looseness=1] (.5,0) to (0,.4);
\draw[blue,in=120,out=90,looseness=1] (.9,.2) to (0,.8);
\draw[red,in=-120,out=-90,looseness=1] (-.5,2) to (0,1.6);
\draw[red,in=-120,out=-90,looseness=1] (-.9,1.8) to (0,1.2);
\draw[blue,in=-90,out=-90,looseness=2] (.5,0) to (.9,0);
\node[scale=.8] at (.7,-.25) {$\color{blue}{\bullet}$};
\node[scale=.8] at (0,.4) {$\color{blue}{\bullet}$};
\node[scale=.8] at (0,.85) {$\color{blue}{\bullet}$};
\node[scale=.8] at (0,1.2) {$\color{red}{\bullet}$};
\node[scale=.8] at (0,1.6) {$\color{red}{\bullet}$};
\end{tikzpicture}} = \raisebox{-8mm}{\begin{tikzpicture}
\draw (0,0) to (0,2);
%\draw[blue] (.9,.2) to (.9,0);
%\draw[blue] (.7,-.3) to (.7,-.7);
\draw[blue,in=120,out=90,looseness=1] (.5,0) to (0,.4);
%\draw[blue,in=120,out=90,looseness=1] (.9,.2) to (0,.8);
\draw[red,in=-120,out=-90,looseness=1] (-.5,2) to (0,1.6);
\draw[red,in=-120,out=-90,looseness=1] (-.9,1.8) to (0,1.2);
%\draw[blue,in=-90,out=-90,looseness=2] (.5,0) to (.9,0);
%\node[scale=.8] at (.7,-.25) {$\color{blue}{\bullet}$};
\node[scale=.8] at (0,.4) {$\color{blue}{\bullet}$};
%\node[scale=.8] at (0,.85) {$\color{blue}{\bullet}$};
\node[scale=.8] at (0,1.2) {$\color{red}{\bullet}$};
\node[scale=.8] at (0,1.6) {$\color{red}{\bullet}$};
\end{tikzpicture}} = \raisebox{-8mm}{\begin{tikzpicture}
\draw (0,-.5) to (0,1.5);
\draw[blue,in=120,out=90,looseness=1] (.5,-.5) to (0,0);
\draw[red,in=-120,out=-90,looseness=1] (-.5,.8) to (0,.5);
\draw[red,in=-90,out=-90,looseness=2] (-.7,1) to (-.3,1);
%\daw[red,in=90,out=90,looseness=2] (-.7,1) to (-1.1,1);
%\draw[red] (-.9,1.2) to (-.9,1.5);
\draw[red] (-.7,1) to (-.7,1.5);
\draw[red] (-.3,1) to (-.3,1.5);
\node[scale=.7] at (-.5,.75) {$\red{\bullet}$};
%\node[scale=.7] at (-.9,1.2) {$\red{\bullet}$};
%\node[scale=.7] at (-.9,1.5) {$\red{\bullet}$};
\node[scale=.7] at (0,.5) {$\red{\bullet}$};
\node[scale=.7] at (0,0) {$\color{blue}{\bullet}$};
\end{tikzpicture}} = 	\raisebox{-6mm}{\begin{tikzpicture}
\draw (-.3,0) to (-.3,-.8);
\draw[blue] (.3,0) to (.3,-.8);
\draw (.3,0) to (.3,.9) ;
\draw[red] (-.3,0) to (-.3,.5);
\draw[red,in=-90,out=-90,looseness=2] (-.6,.9) to (0,.9);
\node[scale=.8] at (-.3,.55) {$\color{red}{\bullet}$};
\node[draw,thick, rounded corners, fill=white,minimum width=30] at (0,0) {$P$};
\end{tikzpicture}}  $$
The second equality is obtained by applying \Cref{associativelem}. Taking adjoint of the second equation of condition (1) of \Cref{qbelmntdefn} and using the separability of the Q-system $B$ gives us the third equality. The fourth equality is obtained from the first equation of condition (1) of \Cref{qbelmntdefn}. Thus, $P$ satisfies the first condition of being a non-unital quantum function.

We also have,
$$\raisebox{-6mm}
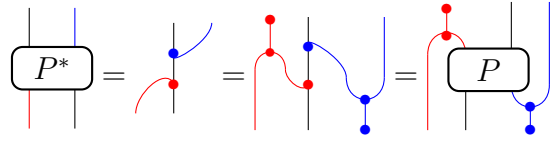
{\begin{tikzpicture}
		\draw[red] (-.3,0) to (-.3,-.8);
		\draw (.3,0) to (.3,-.8);
		\draw[blue] (.3,0) to (.3,.8) ;
		\draw (-.3,0) to (-.3,.8);
		\node[draw,thick, rounded corners, fill=white,minimum width=30] at (0,0) {$P^*$};
\end{tikzpicture}} = \reflectbox{\rotatebox[origin=c]{180}{\raisebox{-4mm}{\begin{tikzpicture}
\draw (0,0) to (0,1.2);
\draw[blue,in=120,out=90,looseness=1] (.5,0) to (0,.4);
\draw[red,in=-120,out=-90,looseness=1] (-.5,1.2) to (0,.8);
\node[scale=.8] at (0,.4) {$\color{blue}{\bullet}$};
\node[scale=.8] at (0,.8) {$\red{\bullet}$};
\end{tikzpicture}}}} = \raisebox{-8mm}{\begin{tikzpicture}
\draw (0,-.3) to (0,1.2);%X
\draw[red][in=-120,out=-90,looseness=1] (-.3,.5) to (0,.3);
\draw[red,in=90,out=90,looseness=2] (-.7,.5) to (-.3,.5);
\draw[red] (-.7,.5) to (-.7,-.3);
\draw[red] (-.5,.7) to (-.5,1.2);
\draw[blue,in=120,out=90,looseness=1] (.5,.4) to (0,.8);
\draw[blue,in=-90,out=-90,looseness=2] (.5,.4) to (1,.4);
\draw[blue] (1,.4) to (1,1.2);
\draw[blue] (.75,-.3) to (.75,.1);
\node[scale=.7] at (-.5,.72) {$\red{\bullet}$};
\node[scale=.8] at (0,.3) {$\color{red}{\bullet}$};
\node[scale=.8] at (-.5,1.15) {$\color{red}{\bullet}$};
\node[scale=.8] at (0,.8) {$\color{blue}{\bullet}$};
\node[scale=.8] at (.75,.1) {$\color{blue}{\bullet}$};
\node[scale=.8] at (.75,-.3) {$\color{blue}{\bullet}$};
\end{tikzpicture}} = \raisebox{-8mm}{\begin{tikzpicture}
\draw (-.3,0) to (-.3,-.8);
\draw[blue,in=-90,out=-90,looseness=2] (.3,-.2) to (.8,-.2);
\draw (.3,0) to (.3,.9) ;
\draw[red,in=90,out=90,looseness=2] (-.3,.2) to (-.8,.2);
\draw[red] (-.55,.45) to (-.55,.8);
\draw[red] (-.8,.2) to (-.8,-.8);
\draw[blue] (.55,-.8) to (.55,-.5);
\draw[blue] (.8,-.2) to (.8,.9);
%\draw[blue,in=-90,out=-90,looseness=2] (-.6,.9) to (0,.9);
\node[scale=.8] at (-.55,.45) {$\color{red}{\bullet}$};
\node[scale=.8] at (-.55,.8) {$\color{red}{\bullet}$};
\node[scale=.8] at (.55,-.5) {$\color{blue}{\bullet}$};
\node[scale=.8] at (.55,-.8) {$\color{blue}{\bullet}$};
\node[draw,thick, rounded corners, fill=white,minimum width=30] at (0,0) {$P$};
\end{tikzpicture}} $$
The second equality follows from condition (3) of \Cref{qbelmntdefn} and its adjoint. Finally, the last equality is obtained from \Cref{associativelem}.
\end{proof}

\begin{defn}\label{intertwinerdefn}
	Let $(H,Q_1,Q_2)$ and $(K,P_1,P_2)$ be two quantum bi-elements of the pair $(A,B)$. We denote $Q_1, Q_2, P_1, P_2$ as follows :
	\[Q_1 \ = \ \raisebox{-2mm}{\begin{tikzpicture}[rotate=180]
			\draw (0,0) to (0,.6);%X
			\draw[red][in=60,out=90,looseness=1] (.5,0) to (0,.3);%Q
			\node[scale=.8] at (0,.3) {$\color{red}{\bullet}$};
	\end{tikzpicture}} \ \ , \ \
	Q_2 \ = \ \raisebox{-2mm}{\begin{tikzpicture}[rotate=180]
			\draw (0,0) to (0,.6);%X
			\draw[blue][in=120,out=90,looseness=1] (-.5,0) to (0,.3);%P
			\node[scale=.8] at (0,.3) {$\color{blue}{\bullet}$};
	\end{tikzpicture}} \ \ , \ \
	 P_1 \ = \ \raisebox{-2mm}{\begin{tikzpicture}[rotate=180]
	\draw[purple] (0,0) to (0,.6);%X
	\draw[red][in=60,out=90,looseness=1] (.5,0) to (0,.3);%Q
	\node[scale=.8] at (0,.3) {$\color{red}{\bullet}$};
\end{tikzpicture}} \ \ , \ \
P_2 \ = \ \raisebox{-2mm}{\begin{tikzpicture}[rotate=180]
\draw[purple] (0,0) to (0,.6);%X
\draw[blue][in=120,out=90,looseness=1] (-.5,0) to (0,.3);%P
\node[scale=.8] at (0,.3) {$\color{blue}{\bullet}$};
\end{tikzpicture}}\]
where $H, K, A$ and $ B $ are represented by black, purple, red and blue strands respectively. An \textit{intertwiner} $f : (H,Q_1,Q_2) \to (K,P_1,P_2)$ is a linear map $f : H \to K $ satisfying the following :
 $$\raisebox{-6mm}{\begin{tikzpicture}[rotate=180]
			\draw (0,0) to (0,.6);%X
			\draw[purple] (0,-.4) to (0,-1);
			\draw[red] (.5,0) to (.5,-1);
			\draw[red][in=60,out=90,looseness=1] (.5,0) to (0,.3);%Q
			\node[scale=.8] at (0,.3) {$\color{red}{\bullet}$};
			\node[draw,thick,rounded corners,fill=white] at (0,-.3) {$f$};
	\end{tikzpicture}} = \raisebox{-6mm}{\begin{tikzpicture}[rotate=180]
	\draw[purple] (0,0) to (0,.6);%X
	\draw (0,1) to (0,1.6);
	\draw[red][in=60,out=90,looseness=1] (.5,0) to (0,.3);%Q
	\node[scale=.8] at (0,.3) {$\color{red}{\bullet}$};
	\node[draw,thick,rounded corners,fill=white] at (0,.9) {$f$};
\end{tikzpicture}} \ \ \ \ \t{and} \ \ \ \ \raisebox{-6mm}{\begin{tikzpicture}[rotate=180]
\draw (0,0) to (0,.6);%X
\draw[purple] (0,-.4) to (0,-1);
\draw[blue] (-.5,0) to (-.5,-1);
\draw[blue][in=120,out=90,looseness=1] (-.5,0) to (0,.3);%P
\node[scale=.8] at (0,.3) {$\color{blue}{\bullet}$};
\node[draw,thick,rounded corners,fill=white] at (0,-.3) {$f$};
\end{tikzpicture}} = \raisebox{-6mm}{\begin{tikzpicture}[rotate=180]
\draw[purple] (0,0) to (0,.6);%X
\draw (0,1) to (0,1.6);
\draw[blue][in=120,out=90,looseness=1] (-.5,0) to (0,.3);%P
\node[scale=.8] at (0,.3) {$\color{blue}{\bullet}$};
\node[draw,thick,rounded corners,fill=white] at (0,.9) {$f$};
\end{tikzpicture}}   $$
\end{defn}

\begin{ex}
	Following \Cref{qbleg}, an intertwiner between qauntum bi-elements of $(A, \C)$ is just an intertwiner (in the sense of \cite{MRV}) of $A$. 
\end{ex}

We now turn our attention to the C*-2-category $\mcal G$ and its Q-system completion. For a pair $(A, B)$ of Q-systems in $\textbf{Hilb}$, we assemble their quantum bi-elements and the corresponding intertwiners into a category, which we will denote by $\textbf{QBL}_{A,B}$.

\begin{defn}
	For a pair $(A, B)$ of Q-systems in $\textbf{Hilb}$, we define the category $\textbf{QBL}_{A,B}$ of quantum bi-elements of the pair $(A, B)$ as follows : 
	\begin{itemize}
		\item [(i)] Objects are quantum bi-elements $(H,Q_1, Q_2)$ of the pair $(A, B)$.
		\item [(ii)] Morphisms from $(H, Q_1, Q_2)$ to $(K,P_1, P_2)$ consist of intertwiners as in \Cref{intertwinerdefn}.
	\end{itemize}
Composition of intertwiners is given by ordinary composition of linear maps.
\end{defn}  

We have observed in \Cref{qsysrem}, that $\mcal G$ is not Q-system complete. We show that the hom categories of $\textbf{QSys}(\mcal G)$ form full subcategories of the category of quantum bi-elements. We summarize this in the following result. 

\begin{thm}
	For two Q-systems $P$ and $Q$ in $\mcal G$, the hom category $\normalfont\textbf{QSys}(\mcal G)(Q, P)$ is a full subcategory of $\normalfont\textbf{QBL}_{Q,P}$.
\end{thm}
\begin{proof}
	This follows from \Cref{frobeniuslem}, \Cref{isoprop} and \cite[Facts 3.15]{CPJP}.
\end{proof}

%\Contact

%%%%%%%%%%%%%%%%%%%%%%%%%%%%%%%%%%%%%%%%%%%%%%%%%

%\bibliographystyle{amsalpha}
%{\footnotesize{
%\bibliography{../bibliography}
%\bibliography{}

\begin{thebibliography}{1}
	
	% \bibitem{Ar} Y. Arano 2014. {\em Unitary spherical representations of Drinfeld doubles}. arXiv:1410.6238
	
	%\bibitem{BiPo} D. Bisch and S. Popa 1998. {\em Examples of subfactors with property T standard invariant}. Geom. Funct. Anal. 9.2, pp. 215-225.
	
	% \bibitem{Bro} A. Brothier and V.F.R. Jones 2015. {\em Hilbert Modules over a Planar Algebra and the Haagerup Property}. arXiv:1503.02708. 
	
	%\bibitem{BH} D. Bisch and U. Haagerup 1996. {\em Composition of subfactors: new examples of infinite depth subfactors}. Ann. scient. Ec. Norm. Sup. 29, pp. 329-383.
	
	%\bibitem{BHP} A. Brothier, D. Penneys and M. Hartglass 2013. {\em Rigid $C^{*}$-tensor categories of bimodules over interpolated free group factors}. arxiv:1208.5505v2.
	
	%\bibitem{BO} N. Brown and N. Ozawa. {\em $C^{*}$-Algebras and Finite Dimensional Approximations}.  Graduate Studies in Mathematics, 88. American Mathematical Society, Providence, Rhode Island, 2008.
	
	%\bibitem{BG} N. Brown and E. Guentner 2012. {\em New $C^{*}$-completions of Discrete Groups and Related Spaces}. arXiv:1205.4649.
	%\vspace{2mm}
%	\bibitem[AMP15] {AMP}\ N.Afzaly, S.Morrison and D.Penneys 2015. The classification of subfactors with index at most $5 \frac{1}{4}$. 
%	https://doi.org/10.48550/arXiv.1509.00038  to appear Mem. Amer. Math. Soc. \vspace{1mm}
	
	%\bibitem[B97]{B97} Dietmar Bisch 1997. Bimodules, higher relative commutants and the fusion algebra associated to a subfactor, Operator algebras and their applications (Waterloo, ON, 1994/1995), 13-63, Fields Inst. Commun., 13, Amer. Math. Soc., Providence, RI, 1997.\vspace{1mm}
	
	
	
	%\bibitem[C20] {Ch}\ Q.Chen 2020. Standard  $\lambda$ -lattices, rigid  $C^*$-tensor categories and bimodules.  arXiv:2009.09273 \vspace{1mm}
	\bibitem[AMRSSV19]{AMRSSV} A. Atserias, L. Mančinska, D. Roberson, R. Šámal, S. Severini, A. Varvitsiotis.	Quantum and non-signalling graph isomorphisms 2019.	\textit{Journal of Combinatorial Theory}, Series B, Volume 136,	2019, Pages 289-328, ISSN 0095-8956, https://doi.org/10.1016/j.jctb.2018.11.002. \vspace*{1mm}
	
	
	
	\bibitem[BBT05]{BBT} G. Brassard, A. Broadbent, and A. Tapp. Quantum pseudo-telepathy 2005. \textit{Foundations of Physics}, 35(11):1877–1907, 2005. arXiv:quant-ph/0407221, doi:10.1007/s10701-005-7353-4. \vspace*{1mm}
	
	\bibitem[CP22]{CP}\ Q.Chen , D.Penneys 2022. Q-system completion is a 3-functor. \textit{Theory and Applications of Categories}, Vol. 38, No. 4, 2022, pp. 101–134. 	arXiv:2106.12437. \vspace*{1mm}
	
	%\bibitem[CPJ22] {CPJ}\ Q.Chen, R.H. Palomares and C.Jones 2022. K-theoretic classification of inductive limit actions of fusion categories on AF-algebras. https://doi.org/10.48550/arXiv.2207.11854
	%\vspace{1mm}	
	
	
	
	\bibitem[CPJP22] {CPJP}\ Q.Chen, R.H. Palomares, C.Jones and D.Penneys 2022.  Q-System completion of C*-2-categories. \textit{Journal of Functional Analysis}, Volume 283, Issue 3, 2022, 109524, ISSN 0022-1236, https://doi.org/10.1016/j.jfa.2022.109524. (https://www.sciencedirect.com/science/article/pii/S0022123622001446)
	\vspace{1mm}
	
	%\bibitem{DGG} P. Das, S. Ghosh and V. Gupta 2012. {\em Drinfeld center of a planar algebra}. arxiv:1203.3958.
	
	%\bibitem[DGGJ22]{DGGJ}\ P.Das, M.Ghosh, S.Ghosh and C.Jones 2022. Unitary connections on Bratteli diagrams. 	arXiv:2211.03822.
	%\vspace*{1mm}
	
	\bibitem[DR18]{DR18}\ C.L. Douglas and D.Reutter 2018. Fusion 2-categories and a state-sum invariant for 4-manifolds. arXiv:1812.11933. \vspace*{1mm}
	
	
	%\bibitem[EK98] {EvKaw}\ D. Evans and Y. Kawahigashi 1998. Quantum symmetries and operator algebras. Oxford Mathematical Monographs, Oxford University Press.\vspace{1mm}
	
	\bibitem[G23a]{G1} M. Ghosh. Q-System Completeness of unitary connections 2023. \textit{Theory and Applications of Categories}, Vol. 39, 2023, No. 37, pp 1121-1151. arXiv:2302.04921. \vspace*{1mm}
	
	\bibitem[G23b]{G2} M. Ghosh. Q-System Completion of $ 2 $-functors 2023. \textit{International Journal of Mathematics}, Vol. 34, No. 12, 2350073 (2023). https://doi.org/10.1142/S0129167X23500738. arXiv:2304.13470. \vspace*{1mm}  
	
	\bibitem[GJF19] {GJF19}\ D.Gaiotto and T.Johnson-Freyd 2019. Condensations in higher categories. arXiv:1905.09566.
	\vspace*{1mm}
	
	
	\bibitem[GLR85]{GLR85} P. Ghez, R. Lima and J. E. Roberts. W*-categories 1985. \textit{Pacific Journal of Mathematics}, 120(1): 79-109 1985. https://doi.org/pjm/1102703884.
	\vspace*{1mm}
	
	\bibitem[HV19]{HV} C. Heunen and J. Vicary. Categories for quantum theory, volume 28 of \textit{Oxford Graduate Texts in Mathematics}. Oxford University Press, Oxford, 2019. An introduction, MR3971584
	DOI:10.1093/oso/9780198739623.001.0001. \vspace*{1mm}
	
	%\bibitem{ENO} P. Etingof, D. Nikshych and V. Ostrik 2005. {\em On fusion categories.} Ann. of Math. 162.2, pp. 581-642.
	
	%\bibitem{GL} J.J. Graham and G.I. Lehrer 1998. {\em The representation theory of affine Temperley Lieb Algebras}.  L'Enseignement Mathematiques 44, pp. 1-44.
	
	%\bibitem{Gh} Shamindra Kumar Ghosh {\em Planar algebras: A category-theoretic point of view}. arxiv:math.QA/0810.4186.
	
	%\bibitem{GhJ} S.K. Ghosh and C.Jones {\em Annular representation theory of rigid, C*-tensor categories.} arxiv: math.OA/1502.06543.
	
	% \bibitem{HI} F. Hiai and M. Izumi 1998. {\em Amenability and strong amenability for fusion algebras with application to subfactor theory}. Internat. J. Math. 9.6, pp. 669-722.
	
	% \bibitem{IMP} M. Izumi, S. Morrison and D. Penneys. {\em Fusion categories between $\mathcal{C}\boxtimes \mathcal{D}$ and $\mathcal{C}*\mathcal{D}$}. arxiv:1308.5723
	
	%\bibitem{I} M. Izumi 1999 {\em The structure of sectors associated with the Longo-Rehren inclusion I. General Theory}. Commun. Math. Phys. 213, pp. 127-179.
	
	%\bibitem{I2} M. Izumi 2001 {\em  The structure of sectors associated with the Longo-Rehren inclusion II. Examples}. Rev. Math. Phys. 13, pp. 603-674.
	
	%\bibitem[I04] {Izm}\  M.Izumi. Non-commutative Markov operators arising from Subfactors. \textit{Advanced Studies in Pure Mathematics}, 2004: 201-217 (2004) 	https://doi.org/10.2969/aspm/03810201 .\vspace{1mm}
	
	
	%\bibitem[J83] {J83}\ V.F.R. Jones 1983. Index for Subfactors. \textit{Invent. Math}. 73, pp. 1-25.  \vspace{1mm}
	
	\bibitem[J99] {J99}\ V.F.R. Jones. Planar Algebras I. arxiv:math.QA/9909027.\vspace{1mm}
	
	%\bibitem[JMS14]{JMS}\ V.Jones, S.Morrison and N.Snyder.  The classification of subfactors of index at most 5. \textit{Bull. Amer. Math. Soc.} (N.S.), 51(2):277–327, 2014. arXiv:1304.6141, doi :10.1090/S0273-0979-2013-01442-3. \vspace{1mm}
	
	%\bibitem[JP19] {JP19}\ C.Jones and D.Penneys. Realizations of algebra objects and discrete subfactors. \textit{Adv.
	%	Math.}, 350:588–661, 2019. MR3948170 DOI:10.1016/j.aim.2019.04.039 . arXiv:1704.02035. \vspace*{1mm}
	
	%\bibitem[JP20] {JP20}\ C.Jones and D.Penneys. Q-systems and compact W*-algebra objects. Topological phases of matter and quantum computation, volume 747 of Contemp. Math., pages 63–88. Amer. Math. Soc., Providence, RI, 2020. MR4079745 DOI:10.1090/conm/747/15039. arXiv:1707.02155. \vspace*{1mm}
	
	\bibitem[JY21] {JY21}  N.Johnson and D.Yau. $2$-dimensional categories. Oxford University Press. https://doi.org/10.1093/oso/9780198871378.001.0001. arXiv:2002.06055.
	\vspace*{1mm}
	
	
	%\bibitem{Ka} C.Kassel {\em Quantum Groups}
	
	\bibitem[KDC16]{KDC} K. D. Commer. Actions of compact quantum groups 2016. arXiv:1604.00159. \vspace*{1mm} 
	
	%\bibitem[L94]{Lon}\ R.Longo 1994. A duality for Hopf algebras and for subfactors. I. \textit{Communications in Mathematical
	%	Physics}, 159(1) 133-150 1994. https://doi.org/cmp/1104254494 . \vspace{1mm}
	
	%\bibitem[LR97]{LR} R. Longo and J.E. Roberts  1997. A theory of dimension. \textit{K-theory} 11(2): pp. 103-159, 1997.
	%\vspace*{1mm}
		
	
	%\bibitem{Mor} S. Morrison.\ {\em A formula for the Jones-Wenzl projections}. arXiv:1503.00384v1
	%\bibitem[M21] {Mon21}\ M.Montgomery 2021. The spectrum of spin model angle operators. arXiv:2110.06720. \vspace{1mm}
	
	\bibitem[M03] {M03}\ M. M{\"u}ger 2003. From subfactors to categories and topology I: Frobenius algebras and Morita equivalence of tensor categories. \textit{Journal of Pure and Applied Algebra}. 180.1, pp. 81-157. \vspace{1mm}
			
	%\bibitem{Mu2} M. M{\"u}ger 2003. {\em From subfactors to categories and topology II: The quantum double of tensor categories and subfactors}. Journal of Pure and Applied Algebra. 180.1, pp. 159-219
	
	\bibitem[MRV18]{MRV}B. Musto, D. Reutter, D. Verdon. A compositional approach to quantum functions 2018. \textit{J. Math. Phys.} 1 August 2018; 59 (8): 081706. https://doi.org/10.1063/1.5020566. arxiv : 1711.07945. \vspace*{1mm} 
	
	\bibitem[N18]{N18} P. Naaijkens. Subfactors and quantum information theory 2018. \textit{Mathematical problems in quantum physics} (2018): 257-279. arXiv:1704.05562. \vspace*{1mm}
	
	\bibitem[NT]{NT} S. Neshveyev, L. Tuset .  Compact Quantum Groups and Their Representation Categories 2014. \textit{Specialized Courses, Vol. 20, SMF} .
	
	%\bibitem[NY17] {NY}\ S. Neshveyev and M. Yamashita 2017.  Poisson boundaries of monoidal categories. \textit{Ann. Sci. Éc. Norm. Supér}. (4) 50 (2017), no. 4, 927-972. https://doi.org/10.24033/asens.2335 . \vspace{1mm}
	
	
	%\bibitem{NY2} S. Neshveyev and M. Yamashita 2015. {\em Drinfeld center and representation theory for monoidal categories}.  Commun. Math. Phys. 345, 385–434 (2016). https://doi.org/10.1007/s00220-016-2642-7 
	
	\bibitem[O88] {O88}\ A. Ocneanu 1988.  Quantized groups, string algebras and Galois theory for algebras. Operator algebras and applications, London Math. Soc. Lecture Note Ser., 136: pp. 119-172. \vspace{1mm} 
	
	% \bibitem{O} V. Ostrik 2003. {\em Module categories, weak Hopf algebras and modular invariants.} Transform. Groups 8.2, pp. 177-206.
	
	% \bibitem{Po0}  S. Popa 1994. {\em Symmetric enveloping algebras, amenability and AFD properties for subfactors}. Math. Res. Lett. 1.4, pp. 409-425.
	%\bibitem[P89] {P89} S. Popa 1989. Relative dimension, towers of projections and commuting squares of subfactors. \textit{Pacific Journal of Mathematics}. Vol. 137 (1989), No. 1, 181–207. \vspace{1mm}
	
	%\bibitem[P94] {P94}\ S. Popa 1994. Classification of amenable Subfactors of type II. \textit{Acta. Math}. 172, pp. 163-225. \vspace{1mm}
	
	\bibitem[P95] {P95}\ S. Popa 1995.  An axiomatization of the lattice of higher relative commutants. \textit{Invent. Math}. 120, pp. 237-252.
	\vspace*{1mm}
	
	%\bibitem{Po3} S. Popa 1997. {\em Amenability in the theory of subfactors}. ``Operator Algebras and Quantum Field Theory”, International Press, editors S. Doplicher et al., pp. 199-211.
	
	%\bibitem{Po4} S. Popa 1999. {\em Some properties of the symmetric enveloping algebra of a subfactor, with applications to amenability and property T}.\ Doc. Math. 4, pp. 665-744.
	
	%\bibitem{Pu} W. Pusz 1993.  {\em Irreducible Unitary Representations of Quantum Lorentz Group}.\ Commun. Math. Phys. 152, pp. 591-626.
	
	%\bibitem{PV} S. Popa and S. Vaes 2014. {\em Representation theory for subfactors, $\lambda$-lattices, and $C^{*}$-tensor categories}.  arXiv:1412.2732v2
	
	%\bibitem[P21]{Po5}\ S.Popa 2021. W*-representations of subfactors and restrictions on the Jones index. https://doi.org/10.48550/arXiv.2112.15148. \vspace{1mm}
	
	
	%\bibitem[S90] {Sch}\  J.K. Schou 1990. PhD thesis. Commuting squares and index for subfactors. https://doi.org/10.48550/arXiv.1304.5907. 
	
	%\bibitem[Z07]{Z07}\ Pasquale A. Zito 2007. 2-C*-categories with non-simple units. \textit{Adv. Math.}, 210(1):122–164, 2007. doi :10.1016/j.aim.2006.05.017. arXiv:math/0509266.
	
	
	
	
	
\end{thebibliography}
\end{document}